\documentclass[11pt,twoside]{amsart}
\usepackage{latexsym,amssymb,amsmath}
\usepackage[all]{xy}
\usepackage{tikz,pgfplots}
\usepackage{extpfeil}
\usepackage{afterpage}
\usepackage{float}
\usepackage{hyperref}

\makeatletter  
\@namedef{subjclassname@2020}{%
\textup{2020} Mathematics Subject Classification}
\makeatother

\numberwithin{equation}{section}
\newtheorem{theorem}{Theorem}[section]
\newtheorem{lemma}[theorem]{Lemma}
\newtheorem{proposition}[theorem]{Proposition}
\newtheorem{corollary}[theorem]{Corollary}
\newtheorem{conjecture}[theorem]{Conjecture}

\theoremstyle{definition}
\newtheorem{definition}[theorem]{Definition} 
\newtheorem{procedure}[theorem]{Procedure} 

\newtheorem{remark}[theorem]{Remark}
\newtheorem{example}[theorem]{Example}

\begin{document}

\title[Symbolic Powers]
{Symbolic powers: Simis and weighted monomial ideals}

\thanks{The first author was supported by a scholarship from
CONAHCYT. 
The second author
was partially supported by the Center for Mathematical Analysis, Geometry and
Dynamical Systems of Instituto Superior T\'ecnico, Universidade de
Lisboa. 
The third author was supported by SNII, Mexico.}

\author[F. O. M\'endez]{Fernando O. M\'endez}
\address{
Departamento de
Matem\'aticas\\
Centro de Investigaci\'on y de Estudios
Avanzados del
IPN\\
Apartado Postal
14--740 \\
Ciudad de M\'exico, M\'exico, CP 07000.
}
\email{fomendez@math.cinvestav.mx}

\author[M. Vaz Pinto]{Maria Vaz Pinto}
\address{Departamento de Matem\'atica, Instituto Superior T\'ecnico,
Universidade de Lisboa, Avenida Rovisco Pais, 1, 1049-001 Lisboa,
Portugal.
} 
\email{vazpinto@math.tecnico.ulisboa.pt}

\author[R. H. Villarreal]{Rafael H. Villarreal}
\address{
Departamento de
Matem\'aticas\\
Centro de Investigaci\'on y de Estudios
Avanzados del
IPN\\
Apartado Postal
14--740 \\
Ciudad de M\'exico, M\'exico, CP 07000.
}
\email{vila@math.cinvestav.mx}


\dedicatory{Dedicated to Professor Sudhir Ghorpade on the occasion of his
$60$th birthday}  

\begin{abstract}
The aim of this work is to compare symbolic and ordinary powers of monomial ideals
using commutative algebra and combinatorics. Monomial ideals whose
symbolic and ordinary powers coincide are 
called Simis ideals. 
Weighted monomial 
ideals are defined by assigning linear weights to monomials. 
We examine Simis and normally torsion-free ideals, 
relate some of the properties of monomial ideals and
weighted monomial ideals, and present a structure theorem for
edge ideals of $d$-uniform clutters whose ideal of covers is Simis in
degree $d$. 
One of our main results is a combinatorial
classification of when the dual of the edge ideal of a weighted oriented
graph is Simis in degree $2$.  
\end{abstract}

\maketitle 

\section{Introduction}\label{intro-section}
Let $S=K[t_1,\ldots,t_s]$ be a polynomial ring over a field $K$ and
let $I$ be an ideal of $S$. 
A prime ideal $\mathfrak{p}$ of $S$ is called an \textit{associated prime}
of $I$ if
$$(I\colon f)=\mathfrak{p},$$
for some $f\in S$, where 
$(I\colon f):=\{g\in S\mid gf\in I\}$ is
an ideal quotient \cite[p.~8]{AM}. The set of associated primes of $I$
is denoted by ${\rm Ass}(I)$. 

Let $I=\bigcap_{i=1}^r\mathfrak{q}_i$ 
be a \textit{minimal primary decomposition} of $I$ with 
${\rm rad}(\mathfrak{q}_i)=\mathfrak{p}_i$, that is, $\mathfrak{q}_i$
is a $\mathfrak{p}_i$-primary ideal, ${\rm
rad}(\mathfrak{q}_i)\neq{\rm
rad}(\mathfrak{q}_j)$ for $i\neq j$ and
$I\neq\textstyle\bigcap_{i\neq j}\mathfrak{q}_{i}$ for
$j=1,\ldots,r$. The set of associated primes of $I$ 
is equal to $\{\mathfrak{p}_1,\ldots,\mathfrak{p}_r\}$
\cite[Theorem~4.5]{AM}.

An associated prime of
$I$ which properly contains another associated prime of $I$ is
called an \textit{embedded prime} of $I$. 
Let ${\rm minAss}(I)$ be the set of minimal elements of ${\rm Ass}(I)$
(minimal with respect to inclusion), that is, the set of non-embedded
associated primes of $I$. The set ${\rm minAss}(I)$ is equal to 
${\rm Min}(I)$, the minimal primes of $I$ \cite{AM}. If
$\mathfrak{p}_i$ is a minimal prime of $I$ and $S_{\mathfrak{p}_i}$ is the
localization of $S$ at $\mathfrak{p}_i$, then $IS_{\mathfrak{p}_i}\cap
S=\mathfrak{q}_i$ \cite[Proposition~4.9]{AM} and 
$\mathfrak{q}_i$ is called the $\mathfrak{p}_i$-primary component of
$I$. If $\mathfrak{p}_i$ is an embedded prime, $\mathfrak{q}_i$ is
not unique and is called an embedded $\mathfrak{p}_i$-primary 
components of $I$. We say $\mathfrak{q}_1,\ldots,\mathfrak{q}_r$ are the
\textit{primary components} of $I$. 

Given an integer $n\geq 1$,  
the $n$-th {\it symbolic power} of 
$I$, denoted $I^{(n)}$, is the ideal  
$$
I^{(n)}:=\bigcap_{\mathfrak{p}\in{\rm minAss}(I)}
(I^nS_{\mathfrak{p}}\textstyle\bigcap
S),
$$
where $I^nS_{\mathfrak{p}}\textstyle\cap
S$ is the ${\mathfrak p}$-primary component of
$I^n$ (see \cite[p.~484]{simis-trung},
\cite[Definition~3.5.1]{Vas1}). In particular, $I^{(1)}$ is the 
intersection of the non-embedded 
primary components of $I$. An alternative
notion of symbolic 
power can 
be introduced 
using the set ${\rm Ass}(I)$ of associated
primes of $I$ instead 
(see, e.g., \cite{cooper-etal,symbolic-powers-survey}): 
$$
I^{\langle n\rangle}: =\bigcap_{\mathfrak{p}\in {\rm
Ass}(I)}(I^nS_\mathfrak{p}{\textstyle\bigcap
S)}=\bigcap_{\mathfrak{p}\in {\rm maxAss}(I)}
(I^nS_\mathfrak{p}\textstyle\bigcap
S),
$$
where ${\rm maxAss}(I)$ denotes the set of maximal elements 
of ${\rm Ass}(I)$ (maximal with respect to inclusion). 
Clearly $I^n\subset I^{\langle n\rangle}\subset I^{(n)}$. 
If $I$ has no embedded primes, the two definitions $I^{(n)}$ 
and $I^{\langle n\rangle}$ of symbolic powers
coincide. 

There are algorithms, that we use in Appendix~\ref{Appendix}, for
computing the symbolic powers of ideals of $S$ which were developed
and implemented in
\textit{Macaulay}$2$ \cite{mac2} by Drabkin, Grifo, 
Seceleanu and Stone \cite{Grifo-etal}. One of these algorithms uses
the methods of  Eisenbud, Huneke, and Vasconcelos for finding 
primary decompositions of ideals of $S$ \cite{EHV}, see also
\cite{singular,Shimoyama} and references therein.  

The containment problem for ordinary and symbolic
powers of ideals consists of determining the positive integers $n$ and
$k$ for which the containment $I^{(n)}\subset I^k$ holds. 
A major result of Hochster and Huneke
\cite[Theorem~1.1]{Hochster-Huneke} shows that 
$I^{(hn)}\subset I^n$ for all positive integers $n$, where 
$h$ is the bigheight of $I$, that is, the largest height of any associated
prime of $I$. We are interested in classifying the containment 
 $I^{(n)}\subset I^n$ for certain families of  ideals.

The monomials of $S$ are denoted by $t^a:=t_1^{a_1}\cdots t_s^{a_s}$,
$a=(a_1,\dots,a_s)\in\mathbb{N}^s$. We say $I$ is a \textit{monomial
ideal} if $I$ is generated by a finite set of monomials of $S$. We
denote the  minimal set of generators of $I$ consisting of monomials by
$\mathcal{G}(I)$. 
The aim of this work is to compare symbolic and ordinary powers of monomial ideals 
using commutative algebra and combinatorics. 

A monomial ideal $I$ is called a \textit{Simis ideal} 
if $I^{(n)}=I^n$ for all $n\geq 1$ and $I$ is called
\textit{normally torsion-free} if ${\rm Ass}(I^n)\subset{\rm Ass}(I)$ 
for all $n\geq 1$. We say $I$ is \textit{Simis} in degree $n$ if
$I^{(n)}=I^n$. The term \textit{Simis ideal} is introduced to recognize the
pioneering work of Aron Simis on symbolic powers of monomial ideals 
\cite{bahiano,HuSV,aron-hoyos,Aron-bookII,simis-ulrich,ITG}. 

Giving a combinatorial
characterization of Simis ideals is a difficult open problem in this
area. This problem has been solved for squarefree monomial ideals,
that is, for edge ideals of clutters 
\cite[Corollary~3.14]{clutters}, \cite[Theorem~1.4]{hhtz}, 
for edge ideals of graphs \cite[Theorem~5.9]{ITG}, for certain classes
of generalized edge ideals \cite[Theorems~3.6 and 3.10]{Das-K}, for ideals of
covers of graphs \cite{alexdual,reesclu}, and for edge ideals of weighted oriented
graphs \cite[Theorem~3.3]{weighted-symbolic}, \cite[Corollary
3.8]{Mandal-Pradhan}, using methods from combinatorial optimization and
graph theory. If one defines symbolic powers of monomial ideals using
all the associated
primes, classifying the Simis ideals seems to be an even more difficult
problem \cite{Banerjee-etal,Das-K,Mandal-Pradhan,Mandal-Pradhan1}. 

We begin to describe the content of this work.
In Section~\ref{prelim-section}, 
we give some more definitions and present some well-known
results. 

Let $\mathfrak{p}$ be a prime ideal containing an ideal $I$ of $S$. The 
\textit{saturation} of $I$ with respect to $\mathfrak{p}$ 
is the ideal $IS_\mathfrak{p}\cap S$. We say $I$ is 
\textit{saturated} with respect to $\mathfrak{p}$ 
when $I=IS_\mathfrak{p}\cap S$. In Section~\ref{symbolic-section}, we
give some sufficient conditions for an ideal to be 
saturated with respect to a prime ideal that are used to 
study symbolic powers (Propositions~\ref{jan16-24} and
\ref{jan21-24-1}). 

In Section~\ref{weighted-section}, we study weighted monomial 
ideals. For convenience, we now introduce this notion. Let 
$w\colon\mathbb{R}^s\rightarrow\mathbb{R}^s$ be a linear function
such that $w(\mathbb{N}_+^s)\subset\mathbb{N}_+^s$, where $\mathbb{N}_+$ denotes the set 
of positive integers. We can write 
$$
w=(w_1,\ldots,w_s),\ a\mapsto(w_1(a),\ldots,w_s(a)), 
$$
where $w_i\colon\mathbb{R}^s\mapsto\mathbb{R}$ is the composition 
$\pi_i w$ of the $i$-th projection map
$\pi_i$ and $w$. 
To any monomial $t^a$ of $S$,  
we assign the weighted monomial $t^{w(a)}=t_1^{w_1(a)}\cdots
t_s^{w_s(a)}$. We call $w$ a \textit{linear weighting} of $S$. A
linear weighting $w$ of $S$ is called \textit{standard} if there are
positive integers $d_1,\ldots,d_s$ such that
$w(a)=(d_1\pi_1(a),\ldots,d_s\pi_s(a))$ for all $a$. Note that in this
case $t^{w(e_i)}=t_i^{d_i}$, where $e_i$ is the $i$-th
unit vector in $\mathbb{R}^s$, and we may 
assign weight $d_i$ to $t_i$, and use the notation
$w(t_i):=d_i$. 

Let $I\subset S$ be a monomial ideal of $S$ and let
$\mathcal{G}(I):=\{t^{v_1},\ldots,t^{v_q}\}$ be the minimal set of
generators of $I$. The \textit{weighted monomial
ideal} of $I$, denoted $I_w$, is given by
$$ 
I_w:=(\{t^{w(a)}\mid t^a\in I\})=(\{t^{w(v_i)}\mid i=1,\ldots,q\}).
$$
\quad An ideal $L$ of $S$ is called {\it irreducible} if 
$L$ cannot be written as an intersection of two ideals of $S$ that
properly contain $L$.  
According to 
\cite[Theorems~6.1.16 and 6.1.17]{monalg-rev}, there is a 
\textit{unique irreducible decomposition}:
\begin{equation}\label{jun4-21}
I=I_{1}\textstyle\cap\cdots\cap I_m,
\end{equation}
such that each ${I}_{i}$ is an irreducible monomial ideal 
of the form ${I}_i=(\{t_k^{b_k}\vert\, b_k\geq 1\})$ for some 
$b=(b_1,\ldots,b_s)$ in
$\mathbb{N}^s\setminus\{0\}$, where $\mathbb{N}=\{0,1,\ldots\}$, 
and $I\neq\textstyle\bigcap_{i\neq j}I_{i}$ for
$j=1,\ldots,m$. The ideals ${I}_{1},\ldots,{I}_m$
are called the {\it irreducible components\/} of $I$. 

Since irreducible ideals are
primary \cite[Lemma~7.12]{AM}, the irreducible decomposition of $I$ is a 
primary decomposition of $I$.  The irreducible decomposition of $I$
is \textit{minimal} if ${\rm rad}(I_i)\neq{\rm
rad}(I_j)$ for $i\neq j$. For edge ideals of
weighted oriented graphs and for squarefree monomial ideals, their irreducible
decompositions are minimal \cite{WOG,monalg-rev}.

Let $I$ be a monomial ideal of $S$ and let $I=I_1\cap\cdots\cap I_m$ 
be the irreducible decomposition of $I$ 
in Eq.~\eqref{jun4-21} and let
$\{\mathfrak{p}_1,\ldots,\mathfrak{p}_r\}$ be the set of all prime 
ideals $\mathfrak{p}$ of $S$ such that ${\rm rad}(I_j)=\mathfrak{p}$
for some $j$. We set 
$$
\mathfrak{q}_i:=\bigcap_{{\rm rad}(I_j)=\mathfrak{p}_i}I_j,\quad
i=1,\ldots,r.
$$
\quad The ideal $\mathfrak{q}_i$ is a $\mathfrak{p}_i$-primary monomial
ideal and this is the irreducible decomposition of $\mathfrak{q}_i$ for
$i=1,\ldots,r$. The following canonical decomposition
\begin{equation}\label{oct28-23}
I=\mathfrak{q}_1\cap\cdots\cap\mathfrak{q}_r
\end{equation}
is a minimal primary decomposition of $I$ and 
${\rm Ass}(I)=\{\mathfrak{p}_1,\ldots,\mathfrak{p}_r\}$. For
monomial ideals, this is the \textit{minimal primary decomposition} 
of $I$ that we use by default.   
By \cite[Lemma~2]{cm-oriented-trees}, one has the following useful 
expression for the $n$-th symbolic power of $I$: 
$$
I^{(n)} =\bigcap_{\mathfrak{p}_i\in {\rm
minAss}(I)}\hspace{-4mm}\mathfrak{q}_i^n\mbox{ for all }n\geq 1. 
$$

Sayedsadeghi and Nasernejad
\cite{nasernejad-w} studied the behavior of normally torsion-freeness
of monomial ideals under standard linear weightings. Other works 
where standard linear weightings are considered are
\cite{Al-Ayyoub-Nasernejad-cover-ideals,Ehrhart-weighting,cm-oriented-trees,
weighted-symbolic,Mandal-Pradhan,Mandal-Pradhan1,lattice-dim1,WOG,nasernejad-s}. 

Taking powers of
ideals commutes with taking weights, that is, if $w$ is a linear 
weighting, then $(I_w)^n=(I^n)_w$ for all $n\geq 1$
(Lemma~\ref{comm-lw}). For standard weightings, this was proved in 
\cite[Lemma~3.5]{nasernejad-w}. The function $I\mapsto I_w$ is
injective for standard weightings (Lemma~\ref{iw=jw}). 

For standard linear weightings it is known that a monomial ideal $I$
is normally torsion-free if and only  
if $I_w$ is normally torsion-free \cite[Theorem~3.10]{nasernejad-w}. 
One can use the following theorem to obtain a similar result for
Simis ideals (Corollary~\ref{simis-w}(b), cf. \cite[Lemma~1]{jtt1}). 


\noindent \textbf{Theorem~\ref{i-iw-thm}.}\textit{\ Let $I$ be a monomial ideal and let 
$I=\mathfrak{q}_1\cap\cdots\cap\mathfrak{q}_r$ be
the minimal primary decomposition of $I$ in Eq.~\eqref{oct28-23}. If
$w$ is a standard weighting of $S$ and $n\in\mathbb{N}_+$ is a fixed
positive integer, 
then $I^n=\bigcap_{i=1}^r\mathfrak{q}_i^n$ 
if and only if
$(I_w)^n=\bigcap_{i=1}^r((\mathfrak{q}_i)_w)^n$.
}

If $I_w$ is a normal ideal (Definition~\ref{normal-ideal-def}) and $w$ is
a standard linear weighting, then $I$ is a normal ideal 
\cite[Theorem~2.6]{Al-Ayyoub-Nasernejad-cover-ideals} (Proposition~\ref{jan24-24}) but the converse does not hold. 

In Section~\ref{symbolic-monomial-section}, we 
give sufficient conditions for a monomial ideal and its powers 
to be saturated with respect to a prime ideal
(Proposition~\ref{jan20-24}). Then, we study Simis and
normally torsion-free monomial
ideals, and relate some of the properties of $I$ and $I_w$.

If 
$\mathfrak{p}$ is an associated prime of an ideal $I$ of $S$ and
$\mathfrak{q}_1,\ldots,\mathfrak{q}_r$ are the primary components of
$I$, with ${\rm rad}(\mathfrak{q}_i)=\mathfrak{p}_i$, 
following \cite{cooper-etal}, we denote the intersection of all $\mathfrak{q}_i$ with
$\mathfrak{p}_i\subset\mathfrak{p}$ by $\mathfrak{q}_{\subset\mathfrak{p}}$. 
By
Proposition~\ref{dec17-23-coro}, one has
$\mathfrak{q}_{\subset\mathfrak{p}}=IS_\mathfrak{p}\cap S$.

As an application, we recover the following result of Cooper, Embree,
H$\rm \grave{a}$ and Hoefel \cite{cooper-etal}: 
If $I\subset S$ is a monomial ideal and $I=\bigcap_{i=1}^r\mathfrak{q}_i$
is a minimal primary decomposition of $I$ with ${\rm
rad}(\mathfrak{q}_i)=\mathfrak{p}_i$, then the $n$-th symbolic power 
$I^{\langle n\rangle}$ of $I$ relative to ${\rm Ass}(I)$ is given by 
$$
I^{\langle n\rangle}=\bigcap_{\mathfrak{p}_i\in {\rm
maxAss}(I)}\hspace{-4mm}(\mathfrak{q}_{\subset\mathfrak{p}_i})^n\ \mbox{ for all }\ 
n\geq 1,
$$
where $\mathfrak{q}_{\subset\mathfrak{p}_i}=IS_{\mathfrak{p}_i}\cap S$
(Corollary~\ref{dec17-23-2}). 
We characterize Simis ideals in algebraic terms 
and note that normally torsion-free ideals are not Simis in
general (Proposition~\ref{ntf-char-coro}, Example~\ref{ntf-simis-minass}). 
For standard weightings, 
$I$ is a Simis ideal if and only if $I_w$ is a Simis ideal
(Corollary~\ref{simis-w}(b)). 

Let $I$ is a monomial ideal and let $I^{\langle n\rangle}$ be its
$n$-th symbolic defined in terms of ${\rm Ass}(I)$. If ${\rm
Ass}(I^n)\subset{\rm Ass}(I)$ for some $n\geq 1$, then 
$I^{\langle n\rangle}=I^n$
(Proposition~\ref{ntf-implies-simis-all-ass}). 
In particular, if $I$ is normally
torsion-free, then $I^{\langle n\rangle}=I^n$ for all $n\geq 1$ but
the converse fails (Example~\ref{converse-ntf-simis-ass}).

For a certain family of ideals, we present a conjecture of what Simis
ideals should be like in terms 
of weightings, and
show some support for the conjecture
(Conjecture~\ref{conjecture-simis}, Corollary~\ref{char-ntf-w}). 
Powers of Simis ideals are Simis
(Proposition~\ref{jan25-24}).

In Section~\ref{symbolic-squarefree}, we show a structure theorem for
edge ideals of $d$-uniform clutters whose ideal of covers is Simis in
degree $d$ (Theorem~\ref{dual-d-simis}) and give another
algebraic classification of bipartite graphs using the second symbolic power of
ideals of covers of graphs (Proposition~\ref{dec11-23}). For
convenience, we now introduce clutters, and their edge ideals and 
ideals of covers.

Let $\mathcal C$ be a \textit{clutter} with vertex 
set $V(\mathcal{C})=\{t_1,\ldots,t_s\}$, that is, $\mathcal C$ is a 
family of subsets $E(\mathcal{C})$ of $V(\mathcal{C})$, called edges,
none of which is contained in
another \cite{cornu-book}. For example, a graph (no multiple edges or loops) is a
clutter. The
 \textit{edge ideal} of $\mathcal{C}$, denoted $I(\mathcal{C})$, 
is the ideal of $S$ given by 
$$I(\mathcal{C}):=(\textstyle\{\prod_{t_i\in e}t_i\mid e\in
E(\mathcal{C})\}).$$
\quad 
The clutter $\mathcal{C}$ is
called $d$-uniform if $|e|=d$ for all $e\in E(\mathcal{C})$. 
A set of vertices $C$ of $\mathcal{C}$ is called a \textit{vertex
cover} if every edge of $\mathcal{C}$ contains at least one vertex of
$C$. A  \textit{minimal vertex cover} of $\mathcal{C}$ is a vertex cover which
is minimal with respect to inclusion. The ideal of covers of
$\mathcal{C}$, denoted $I_c(\mathcal{C})$, is generated by all
$\prod_{t_i\in C}t_i$ such that $C$ is a minimal vertex cover of
$\mathcal{C}$ \cite[p.~221]{monalg-rev}.

The paper of Huneke,
Simis and Vasconcelos \cite{HuSV} was first used in \cite{Lisboar,reesclu,clutters} to study 
combinatorial problems of clutters using symbolic powers of 
edge ideals of clutters, see also \cite{Seceleanu-packing} and 
the survey papers \cite{symbolic-powers-survey,wolmer-survey}.
A breakthrough in the area of edge ideals is a theorem relating 
symbolic powers and the max-flow min-cut property of integer
programming \cite[Corollary~3.14]{clutters}, cf. \cite[Theorem
1.4]{hhtz}, creating a bridge between algebra and optimization
problems. A famous conjecture of Conforti--Cornu\'ejols
\cite{CC} from combinatorial optimization, 
known as the packing problem for clutters, 
was shown to be equivalent to the equality of ordinary and symbolic
powers of edge ideals \cite[Conjecture~3.10]{clutters}, 
\cite[Theorem~4.6]{reesclu}. To the best of our knowledge the 
conjecture is still unsolved. 
 
We come to one of our main results.

\noindent \textbf{Theorem~\ref{dual-d-simis}.}\textit{\ 
Let $\mathcal{C}$ be a $d$-uniform clutter without isolated vertices. If
$I_c(\mathcal{C})^{(d)}=I_c(\mathcal{C})^{d}$, then there are 
mutually disjoint minimal vertex covers $C_1,\ldots,C_d$ of
$\mathcal{C}$ such that $V( \mathcal{C})=\bigcup_{i=1}^d C_i$ and
every edge of $\mathcal{C}$ has the form 
$e=\{t_{i_1},\ldots,t_{i_d}\}$, where $t_{i_j}\in C_j$ for all $j$.
}

The converse of Theorem~\ref{dual-d-simis} holds if $d=2$, that is, if
the clutter $\mathcal{C}$ is a graph 
(Proposition~\ref{dec11-23}) but it fails in general
(Example~\ref{dual-d-simis-ce}). If $G$ is a graph and $I_c(G)$ is
its 
ideal of covers, we show that 
$I_c(G)^{(2)}=I_c(G)^2$ if and only if $G$ is bipartite 
(Proposition~\ref{dec11-23}). The associated primes of $I_c(G)^2$ 
were determined by  
Francisco, H$\rm \grave{a}$ and Van Tuyl \cite{fhv}, they
showed that $\mathfrak{p}$ is an associated prime of $I_c(G)^2$ if an
only if $\mathfrak{p}$ is generated by the vertices of an edge of $G$
or by  the vertices forming an induced 
odd cycle of $G$. The minimal generators of
$\mathcal{R}_s(I_c(G))=\bigoplus_{n=0}^{\infty}I_c(G)^{(n)}t^n$, the
symbolic Rees algebra of $I_c(G)$, were determined by Dupont and the
third author \cite{covers}. We can use these results to give
alternative proofs of Proposition~\ref{dec11-23}. 

In Section~\ref{symbolic-dual-section}, we classify combinatorially 
when the dual of
the edge ideal of a weighted oriented graph is a Simis ideal in degree
$2$ (Theorem~\ref{J2=J(2)}). For convenience we briefly 
introduce weighted oriented graphs and the dual of their edge ideals.   

Let $G$ be a simple graph with vertex set $V(G)=\{t_1,\ldots,t_s\}$ and edge
set $E(G)$. Let $D$ be a {\it weighted oriented graph\/} whose {\it underlying
graph\/} is $G$, that is, $D$ is a triplet $(V(D),E(D),w)$ where $V(D)=V(G)$,
$E(D)\subset V(D)\times V(D)$ such that $E(G)$ is the set of
pairs $\{t_i,t_j\}$ with $(t_i,t_j)\in E(D)$,  
$|E(D)|=|E(G)|$, and  
$w\colon V(D) \to\mathbb{N}_+$ is a \textit{weight function}.  In other words, $D$ is obtained
from $G$ by assigning a direction to its edges and a weight to its
vertices. The \textit{vertex set} of $D$ and the \textit{edge set} of $D$
are $V(D)$ and $E(D)$, respectively. The \emph{weight} of $t_i\in V(D)$
is $w(t_i)$ and is denoted simply by $w_i$. The set of vertices
$\{t_i\in V(D)\mid w_i>1\}$ is denoted by
$V^{+}(D)$. The
 \textit{edge ideal} of $D$, introduced in \cite{cm-oriented-trees,WOG}, 
is the ideal of $S$ given by 
$$I(D):=(\{t_{i}t_{j}^{w_j}\mid (t_{i},t_{j})\in E(D)\}).$$
\quad If $w_i=1$ for each $t_i\in V(D)$, then $I(D)$ is the usual edge
ideal $I(G)$ of the graph $G$ \cite{cm-graphs} The
motivation to study $I(D)$ comes from coding theory, see
\cite[p.~536]{oriented-graphs} and \cite[p.~1]{WOG}. If a vertex $t_i$ of
$D$ is a {\it source\/} (i.e., a vertex 
with only outgoing edges) we shall always 
assume that $w_i=1$ because in this case the definition of $I(D)$
 does not depend on the weight of $t_i$ (Remark~\ref{nov30-21}). A
 \emph{sink} vertex of $D$ is a vertex with only incoming edges. This
 notion will play a role in some of our main results. If all
 vertices of $V^+(D)$ are sinks, then $I(D)$ is obtained from $I(G)$ 
by making the change of variables $t_i\rightarrow t_i^{w_i}$ for
$i=1,\ldots,s$.

Following
\cite[p.~495]{cm-oriented-trees}, we define the 
\textit{dual} $J(D)$ of the edge ideal $I(D)$ as: 
$$
J(D):=\bigcap_{(t_i,t_j)\in E(D)}
(t_i,t_j^{w_j}),
$$
and this is the irreducible decomposition of $J(D)$. 
If $w_i=1$ for
all $i$, then $J(D)$ is the ideal of covers $I_c(G)$ of $G$ and we
denote $J(D)$ by $J(G)$.  

If $G$ is a graph, 
then $J(G)$ is Simis if and only if $G$ is bipartite
\cite{alexdual}, \cite[Corollary~3.17, Theorem~4.6,
Proposition~4.27]{reesclu}. 
By \cite[Theorem~3.2]{weighted-symbolic}, 
$I(D)^2=I(D)^{(2)}$ if and only if 
every vertex in $V^+(D)$ is a sink and the underlying graph $G$ of
$D$ has no triangles. 

We come to our main result.

\noindent \textbf{Theorem~\ref{J2=J(2)}.}\textit{\ 
Let $D$ be a weighted oriented graph and let $J(D)$ be the dual of 
the edge ideal $I(D)$ of $D$. Then, $J(D)^2=J(D)^{(2)}$ if and only if the following
two conditions hold:
\begin{enumerate}
\item[{\rm (i)}] Every vertex in $V^+(D)$ is a sink.
\item[{\rm (ii)}] The underlying 
graph $G$ of $D$ is bipartite.
\end{enumerate}
}

The ideal $I(D)$ is Simis if and only if every vertex in $V^+(D)$ is
a sink and $G$ is a bipartite graph
\cite[Theorem~3.3]{weighted-symbolic}, \cite[Corollary
3.8]{Mandal-Pradhan}. If $J(D)^2=J(D)^{(2)}$, then
$I(D)^2=I(D)^{(2)}$ (Theorems~\ref{J2=J(2)} and \ref{In=I(n)}) but the
converse does not hold (Example~\ref{converseI2}).

If the ideal $I(D)$ has no embedded primes, we give
some classifications of the normally torsion-freeness of $J(D)$ 
(Corollary~\ref{ntf-weighted}).
The normality of $J(D)$ is related to the Simis property of $J(D)$. 
If $J(D)^n=J(D)^{(n)}$ for all $n\geq
1$, then $J(D)$ is a normal ideal
\cite[Corollary~4.2]{weighted-symbolic}. If $I$ is the edge ideal of
a clutter and $I$ is Simis, then $I$ is normal \cite[Corollary~5.3]{ITG}.

In Section~\ref{examples-section}, we present examples related to 
some of our results. Then, in Appendix~\ref{Appendix}, we give the
procedures for \textit{Macaulay}$2$ \cite{mac2} that are used in the examples 
to compute symbolic powers, primary and irreducible
decompositions, and saturations of monomial ideals.  

For unexplained
terminology and additional information,  we refer to 
\cite{graphs,Herzog-Hibi-book,edge-ideals,Sta5,chapter-vantuyl,monalg-rev}
for the theory of edge ideals and 
\cite{AM,Eisen} for 
commutative algebra.   

\section{Preliminaries}\label{prelim-section}
In this section we give some definitions and present some well-known
results that will be used in the following sections.  
To avoid repetitions, we continue to employ 
the notations and
definitions used in Section~\ref{intro-section}.

\begin{definition}
An ideal $\mathfrak{q}$ in $S$ is \textit{primary} if
$\mathfrak{q}\neq S$ and if $xy\in \mathfrak{q}$, $x,y\in S$, implies
that either $x\in\mathfrak{q}$ of $y^n\in\mathfrak{q}$ for some 
$n\geq 1$.
\end{definition}

\begin{proposition}\cite[Proposition~4.8]{AM}\label{dec17-23}
Let $R$ be a multiplicatively closed subset of $S$ and let
$\mathfrak{q}$ be a $\mathfrak{p}$-primary ideal. The following hold.

{\rm(i)} If $R\cap \mathfrak{p}\neq\emptyset$, then
$R^{-1}\mathfrak{q}=R^{-1}S$.

{\rm(ii)} If $R\cap\mathfrak{p}=\emptyset$, then $R^{-1}\mathfrak{q}$ is
$R^{-1}\mathfrak{p}$-primary and its contraction in $S$ is
$\mathfrak{q}$.
\end{proposition}

\begin{proposition}\cite[Proposition~4.9]{AM}\label{dec17-23-coro}
Let $I\subset S$ be an ideal and let $I=\bigcap_{i=1}^r\mathfrak{q}_i$
be a minimal primary decomposition of $I$ with ${\rm
rad}(\mathfrak{q}_i)=\mathfrak{p}_i$. If $\mathfrak{p}\in {\rm
Ass}(I)$, then 
$$
IS_\mathfrak{p}=\bigcap_{\mathfrak{p}_i\subset\mathfrak{p}}
\mathfrak{q}_iS_\mathfrak{p}\ \mbox{ and}\ 
IS_\mathfrak{p}\cap S=\bigcap_{\mathfrak{p}_i\subset\mathfrak{p}}
\mathfrak{q}_i,
$$
and these are minimal primary decompositions.
\end{proposition}

\begin{proof} Setting $R=S\setminus\mathfrak{p}$, by
Proposition~\ref{dec17-23}, one has: 
{\rm(i)} if $\mathfrak{p}_i\not\subset\mathfrak{p}$, then
$\mathfrak{q}_iS_\mathfrak{p}=S_\mathfrak{p}$, and {\rm(ii)} 
if $\mathfrak{p}_i\subset\mathfrak{p}$, then
$\mathfrak{q}_iS_\mathfrak{p}$ is
$\mathfrak{p}_iS_\mathfrak{p}$-primary and
$\mathfrak{q}_iS_\mathfrak{p}\cap S=\mathfrak{q}_i$. Hence, the
result follows by localizing the primary decomposition of 
$I$ at $\mathfrak{p}$.
\end{proof}

\begin{proposition}
An ideal $\mathfrak{q}$ in $S$ is {primary} if and only 
if ${\rm rad}(\mathfrak{q})$ is prime and 
$\mathfrak{q}S_\mathfrak{p}\cap S=\mathfrak{q}$, where 
$\mathfrak{p}={\rm rad}(\mathfrak{q})$.
\end{proposition}

\begin{proof} $\Rightarrow$) From the 
definition of primary ideal, we get that ${\rm rad}(\mathfrak{q})$ 
is prime. The equality $\mathfrak{q}S_\mathfrak{p}\cap
S=\mathfrak{q}$ follows from Proposition~\ref{dec17-23-coro}. 

$\Leftarrow$)  Take $x,y\in S$ such that $xy\in\mathfrak{q}$ and
assume that $x\notin\mathfrak{q}$. If $y\notin\mathfrak{p}$, then
$x=(xy)/y\in \mathfrak{q}S_\mathfrak{p}\cap
S$, and consequently $x\in\mathfrak{q}$, a contradiction. 
Thus, $y\in \mathfrak{p}={\rm rad}(\mathfrak{q})$, and
$y^n\in\mathfrak{q}$ for some $n\geq 1$. 
\end{proof}

The \textit{support} of a monomial $t^a=t_1^{a_1}\cdots t_s^{a_s}$, denoted 
${\rm supp}(t^a)$, is the set of all $t_i$ such that $a_i\neq 0$. 
The next result has no analogue for 
graded primary ideals of $S$ (Example~\ref{jan21-24}).

\begin{proposition}\cite[Proposition~6.1.7]{monalg-rev}\label{stanprim-iff}
A monomial ideal $\mathfrak{q}\subset S$ is 
primary if and only if, after permutation of the 
variables $t_1,\ldots,t_s$ in $S$, $\mathfrak{q}$ has the form:
\[
\mathfrak{q}=(t_1^{a_1},\ldots,t_k^{a_k},t^{b_1},\ldots,t^{b_\ell}),
\]
where $a_i\geq 1$ and 
$\cup_{i=1}^\ell{\rm supp}(t^{b_i})\subset \{t_1,\ldots,t_k\}$. 
\end{proposition}

\begin{lemma}{\rm(\cite[Lemma~2]{cm-oriented-trees},
\cite[Lemma~3.1]{herzog-hibi-trung})} 
\label{anoth-one-char-spow-general} 
Let $I$ be a monomial ideal of $S$. If 
$\mathfrak{q}_1,\ldots,\mathfrak{q}_p$ are the primary components
corresponding to the  minimal primes 
of $I$, then 
$$
I^{(n)}={\mathfrak q}_1^{n}\textstyle\cap\cdots\cap {\mathfrak q}_p^{n}\ 
\mbox{ for all }\ n\geq 1.
$$
\end{lemma}

\begin{proof} By Proposition~\ref{stanprim-iff}, it follows that $\mathfrak{q}_i^n$ is
a $\mathfrak{p}_i$-primary ideal for all $i$. Then, applying 
Propositions~\ref{dec17-23-coro} and \ref{dec17-23}, one has
$$
I^{(n)}=\bigcap_{i=1}^p(I^nS_{\mathfrak{p}_i}\cap
S)=\bigcap_{i=1}^p((IS_{\mathfrak{p}_i})^n\cap S)
=\bigcap_{i=1}^p((\mathfrak{q}_iS_{\mathfrak{p}_i})^n\cap S)=
\bigcap_{i=1}^p(\mathfrak{q}_i^nS_{\mathfrak{p}_i}\cap
S)=\bigcap_{i=1}^p\mathfrak{q}_i^n,
$$
and the proof is complete.
\end{proof}

\begin{lemma}{\rm(cf.~\cite[Proposition 3.8]{nasernejad-w})} Let $L,I_1,\ldots,I_m$ be
irreducible primary monomial ideals of $S$. If\/
$\bigcap_{i=1}^mI_i\subset L$, then $I_i\subset L$ for some $i$.
\end{lemma}

\begin{proof} This follows from the proofs of
\cite[Proposition~6.1.16, Theorem~6.1.17]{monalg-rev}.
\end{proof}

\begin{lemma}\cite[p.~208]{monalg-rev}\label{intofmonoismono} 
Let $I$ and $J$ be two ideals generated by finite sets of 
monomials $\mathcal{G}(I)$ and $\mathcal{G}(J)$, respectively. 
Then, the intersection $I\cap J$ is generated by 
the set 
$$
\{{\rm lcm}(f,g)\mid f\in \mathcal{G}(I)\mbox{ and } g\in
\mathcal{G}(J)\},
$$
where ${\rm lcm}(f,g)$ denotes the least common multiple of $f$ and $g$.
\end{lemma} 

\begin{remark}\label{nov30-21} 
Let $D = (V(D),E(D),w)$ be a weighted oriented graph
with vertex set $V(D)=\{t_1,\ldots,t_s\}$, underlying graph $G$, and
edge ideal $I(D)$.
Consider the weighted
oriented graph $D' = (V(D),E(D),w')$ with $w'(t_i) = 1$ if $t_i$ is a
\textit{source} vertex and $w'(t_i) =
w(t_i)$ if $t_i$ is not a source vertex. 
Then, $I(D') = I(D)$, that is, $I(D)$ does not depend on the weights
that we place at source vertices. For this reason we will always 
assume that all sources of $D$ have weight $1$.
\end{remark}

\begin{lemma}\cite[p.~169]{Vas1}\label{icd}
If $I$ is a monomial ideal of $S$ and $n\in\mathbb{N}_+$, then the
integral closure $\overline{I^n}$ of $I^n$ is given by 
\begin{equation*}
\overline{I^n}=(\{t^a\in S\mid (t^a)^{p}\in I^{pn}
\mbox{ for some }p\geq 1\}).
\end{equation*}
\end{lemma}

\begin{definition}\label{normal-ideal-def}
A monomial ideal $I$ is \textit{normal} if $\overline{I^n}=I^n$ for 
all $n\geq 1$. 
\end{definition}
 
\section{Symbolic powers of ideals}\label{symbolic-section}
In this section, we give some sufficient conditions for an ideal to be
saturated with respect to a prime ideal, that will be used to 
study symbolic powers. To avoid repetitions, we continue to employ 
the notations and definitions used in Sections~\ref{intro-section}
and 
\ref{prelim-section}.

\begin{proposition}\label{jan16-24} Let $I\subset S$ be an ideal and
let $\mathfrak{p}\subset S$ be a prime ideal. If $\mathfrak{p}$ 
contains all associated primes of $I$, then $IS_\mathfrak{p}\cap S=I$.
\end{proposition} 

\begin{proof} Clearly $IS_\mathfrak{p}\cap S\supset I$. To show
the reverse inclusion take $f\in IS_\mathfrak{p}\cap S$. Then, $f=g/h$,
$g\in I$, $h\notin\mathfrak{p}$. Thus, $hf\in I$. Pick a minimal primary
decomposition $I=\bigcap_{i=1}^\ell Q_i$, where $Q_i$ is a
$P_i$-primary ideal for $i=1,\ldots,\ell$. Then, $hf\in\ Q_i$ for
all $i$. If $f\notin Q_i$ for some $i\in\{1,\ldots,\ell\}$, 
then $hf\in Q_i$ and consequently $h^p\in Q_i$ for some $p\geq 1$.
Hence, $h\in P_i\subset\mathfrak{p}$, a contradiction. This proves that $f\in Q_i$ for
all $i$, that is, $f\in I$.
\end{proof}

\begin{corollary}\cite[Lemma~2.13]{weighted-symbolic}\label{feb11-24}
Let $I\subset S$ be a graded
ideal. The following hold.
\begin{enumerate}
\item[(a)] If $\mathfrak{m}=(t_1,\ldots,t_s)$ is the irrelevant
maximal ideal of $S$, then $IS_\mathfrak{m}\cap S=I$;
\item[(b)] If $\mathfrak{m}\in{\rm Ass}(I)$, then 
$I^{\langle n\rangle}=I^nS_\mathfrak{m}\cap S=I^n$ for all $n\geq 1$.
\end{enumerate}
\end{corollary} 

\begin{proof} (a) Pick a minimal primary
decomposition $I=\bigcap_{i=1}^\ell Q_i$, where $Q_i$ is a
graded $P_i$-primary ideal for $i=1,\ldots,\ell$
\cite[Theorem~2.2.8]{monalg-rev}. Then,
$Q_i\subset\mathfrak{m}$  and consequently
$P_i\subset\mathfrak{m}$ for all $i$. Hence, by
Proposition~\ref{jan16-24}, one has $IS_\mathfrak{m}\cap S=I$.

(b) Note that ${\rm maxAss}(I)=\{\mathfrak{m}\}$. Then, by part (a), we
get $I^{\langle n\rangle}=I^nS_\mathfrak{m}\cap 
S =I^n$.
\end{proof}

\begin{proposition}\label{jan21-24-1} Let $I\subset S$ be an ideal
generated by polynomials $f_1,\ldots,f_q$ in $K[V]$ for 
some $V\subset\{t_1,\ldots,t_s\}$ and let 
$\mathfrak{N}=(V)$ be the ideal of $S$ generated by $V$. 
If all the associated primes of $I$ are generated by monomials, then 
any associated prime of $I$ is contained in
$\mathfrak{N}$ and $IS_\mathfrak{N}\cap S=I$. 
\end{proposition}

\begin{proof} Let $\mathfrak{p}$ be an associated prime of $I$, that is, 
$\mathfrak{p}=(I\colon f)$ for some $f\in S$. Note that
$\mathcal{G}(\mathfrak{p})$, the minimal generating set of
$\mathfrak{p}$, is a subset of $\{t_1,\ldots,t_s\}$
because $\mathfrak{p}$ is a monomial prime ideal of $S$. To show
$\mathfrak{p}\subset\mathfrak{N}$, it suffices to show
the inclusion $\mathcal{G}(\mathfrak{p})\subset V$ because $V\subset\mathfrak{N}$.
We argue by
contradiction assuming that there is
$t_\ell\in\mathcal{G}(\mathfrak{p})\setminus V$. Then, $t_\ell f\in I$ and we can
write
$$
t_\ell f=(a_{1,1}+t_\ell a_{1,2})f_1+\cdots+(a_{q,1}+t_\ell
a_{q,2})f_q,
$$
where $a_{i,1}\in
K[\{t_1,\ldots,t_s\}\setminus\{t_\ell\}]$ and
$a_{i,2}\in S$ for $i=1,\ldots,q$. Hence, making $t_\ell=0$, we get 
$\sum_{i=1}^qa_{i,1}f_i=0$ and consequently
$f=\sum_{i=1}^qa_{i,2}f_i$. Hence, $f\in(f_1,\ldots,f_q)=I$ and
$1\in\mathfrak{p}$, a contradiction. Thus,
$\mathfrak{p}\subset\mathfrak{N}$. Therefore, by Proposition~\ref{jan16-24}, one has 
$IS_\mathfrak{N}\cap S=I$. 
\end{proof}

\section{Weighted monomial ideals}\label{weighted-section}
In this section, we study weighted monomial 
ideals. To avoid repetitions, we continue to employ 
the notations and definitions used in Sections~\ref{intro-section}
and 
\ref{prelim-section}.

\begin{lemma}\cite[Lemma~3.5]{nasernejad-w}\label{i-iw} 
Let 
$w$ be a linear weighting of $S$. 
If $I_1,\ldots,I_m$ are irreducible monomial ideals of $S$ and 
$I=I_1\cap\cdots\cap I_m$, then 
$$   
I_w\subset (I_1)_w\cap\cdots\cap (I_m)_w, 
$$
with equality if $w$ is a standard weighting of $S$.  
\end{lemma}

\begin{proof} To show the inclusion ``$\subset$'' take 
$t^{w(a)}\in I_w$ with $t^a\in I$. As $t^a\in I_i$ for all $i$, 
$t^{w(a)}\in(I_i)_w$ for all $i$, that is,
$t^{w(a)}\in\bigcap_{i=1}^m(I_i)_w$. 
To show the second part assume that $w$ is a standard weighting, that
is, there are
positive integers $n_1,\ldots,n_s$ such that
$w(a)=(n_1\pi_1(a),\ldots,n_s\pi_s(a))$ for all $a$. Take $t^a\in(I_1)_w\cap\cdots\cap (I_m)_w$. 
Then
\begin{equation}\label{oct25-23}
t^a=t^{\delta_1}t_{j_1}^{\ell_1n_{j_1}}=\cdots=t^{\delta_m}t_{j_m}^{\ell_mn_{j_m}}
\ \mbox{ and }\ t^a=t^\delta{\rm
lcm}\{t_{j_1}^{\ell_1n_{j_1}},\ldots,t_{j_m}^{\ell_rn_{j_m}}\},
\end{equation}
where $t^{\delta_i},t^\delta\in S$ and
$t_{j_i}^{\ell_i}\in\mathcal{G}(I_i)$ for all $i$. We may assume that
$j_1,\ldots,j_k$ are distinct and $j_i\in\{j_1,\ldots,j_k\}$ for
$i>k$. We may also assume that 
$$
\ell_p=\max\{\ell_i\mid 1\leq i\leq
m,\, j_i=j_p\}
$$
for $1\leq p\leq k$. Then, setting 
$c=\ell_1e_{j_1}+\cdots+\ell_ke_{j_k}$, one has 
\begin{equation} 
t^c=t_{j_1}^{\ell_1}\cdots t_{j_k}^{\ell_k}={\rm
lcm}\{t_{j_1}^{\ell_1},\ldots,t_{j_m}^{\ell_m}\}\in I\ \mbox{ and }\ 
t^{w(c)}=t_{j_1}^{\ell_1n_{j_1}}\cdots
t_{j_k}^{\ell_kn_{j_k}}\in I_w.
\end{equation}
\quad Therefore, from the equalities
$$
t_{j_1}^{\max\{\ell_in_{j_i}\mid 1\leq i\leq m,\,
j_i=j_1\}}=t^{\epsilon_1}t_{j_1}^{\ell_1n_{j_1}},\ldots,
t_{j_k}^{\max\{\ell_in_{j_i}\mid 1\leq i\leq m,\,
j_i=j_k\}}=t^{\epsilon_k}t_{j_k}^{\ell_kn_{j_k}},
$$
we get that ${\rm
lcm}\{t_{j_1}^{\ell_1n_{j_1}},\ldots,t_{j_m}^{\ell_mn_{j_m}}\}=
(t^{\epsilon_1}t_{j_1}^{\ell_1n_{j_1}})\cdots(t^{\epsilon_k}t_{j_k}^{\ell_kn_{j_k}})=
t^\epsilon t^{w(c)}$. Then, by Eq.~\eqref{oct25-23}, we obtain that
$t^a$ is a multiple of $t^{w(c)}$, and consequently $t^a\in I_w$.
\end{proof}

\begin{proposition}\cite{nasernejad-w}\label{i-iw-lem} If $w$ is a standard weighting and 
$J_1,\ldots,J_r$ are  monomial ideal, then 
$$
(J_1\cap\cdots\cap J_r)_w=({J}_1)_w\cap\cdots\cap({J}_r)_w.
$$
\end{proposition}

\begin{proof} Let $J_j=\bigcap_{i=1}^{\ell_j} J_{i,j}$ be the
irreducible decomposition of $J_j$ for $j=1,\ldots,r$. Then, by
applying Lemma~\ref{i-iw} twice, one has
\begin{align*}
(J_j)_w=&\bigcap_{i=1}^{\ell_j}(J_{i,j})_w\ \mbox{ for }\
j=1,\ldots,r,\ \mbox{ and }\\
\bigg(\bigcap_{j=1}^r J_j
\bigg)_w&=\bigg(\bigcap_{j=1}^r\bigg(\bigcap_{i=1}^{\ell_j}J_{i,j}\bigg)\bigg)_w
=\bigg(\bigcap_{i,j}J_{i,j}\bigg)_w=\bigcap_{i,j}(J_{i,j})_w=
\bigcap_{j=1}^r\bigg(\bigcap_{i=1}^{\ell_j}(J_{i,j})_w\bigg)
=\bigcap_{j=1}^r (J_j)_w,
\end{align*}
and the proof is complete.
\end{proof}

\begin{lemma}\cite{nasernejad-w}\label{comm-lw}
If $w$ is a linear weighting and $I$ is a monomial ideal, then 
$(I_w)^n=(I^n)_w$.
\end{lemma}

\begin{proof} To show the inclusion ``$\subset$'', take
any monomial $t^a\in(I_w)^n$, that is, 
$t^a=t^\delta t^{w(\alpha_1)}\cdots t^{w(\alpha_n)}$, where
$t^\delta\in S$ and $t^{\alpha_i}\in I$ for all $i$. Setting
$\beta=\alpha_1+\cdots+\alpha_n$, one has $t^\beta\in I^n$,
$t^{w(\beta)}\in(I^n)_w$, and $t^a=t^\delta t^{w(\beta)}$. 
Thus, $t^a\in(I^n)_w$. 
To show the reverse inclusion ``$\supset$'' take any monomial $t^a\in(I^n)_w$, that is, 
$t^a=t^\epsilon t^{w(\gamma)}$, where $t^\epsilon\in S$ and $t^\gamma\in I^n$. Then, we can
write $t^\gamma=t^{b_1}\cdots t^{b_n}$, where $t^{b_i}\in I$ for all
$i$, and consequently $t^{w(\gamma)}\in(I_w)^n$. Thus,
$t^a=t^\epsilon t^{w(\gamma)}\in(I_w)^n$.
\end{proof}

\begin{lemma}\label{iw=jw}
Let $w$ be a standard weighting of $S$. If $I$ and $J$ are monomial
ideals of $S$ and $I_w=J_w$, then $I=J$.
\end{lemma}

\begin{proof} To show the inclusion ``$\subset$'' take $t^a\in I$.
Then, $t^{w(a)}\in I_w$ and we can write $t^{w(a)}=t^\delta t^{w(b)}$
for some $t^\delta\in S$ and $t^b\in J$. Thus, $w(a)=\delta+w(b)$. We
set $a=(a_1,\ldots,a_s)$, $b=(b_1,\ldots,b_s)$. Since $w$ is standard,
$w(a)=(\pi_1(a)n_1,\ldots,\pi_s(a)n_s)$, $n_i\in\mathbb{N}_+$ for all
$i$, and we have $a_in_i\geq b_in_i$ for all $i$. Then, we can write
$a=\epsilon+b$ for some $\epsilon\in\mathbb{N}^s$, and consequently
$t^a=t^\epsilon t^b\in J$. The inclusion ``$\supset$'' follows using
similar arguments.
\end{proof}


\begin{theorem}\label{i-iw-thm} Let $I$ be a monomial ideal and let 
$I=\mathfrak{q}_1\cap\cdots\cap\mathfrak{q}_r$ be
the minimal primary decomposition of $I$ in Eq.~\eqref{oct28-23}. If
$w$ is a standard weighting of $S$ and $n\in\mathbb{N}_+$ is a fixed
positive integer, 
then $I^n=\bigcap_{i=1}^r\mathfrak{q}_i^n$ 
if and only if
$(I_w)^n=\bigcap_{i=1}^r((\mathfrak{q}_i)_w)^n$.
\end{theorem}

\begin{proof} $\Rightarrow$) Assume the equality 
$I^n=\bigcap_{i=1}^r\mathfrak{q}_i^n$ for some $n\in\mathbb{N}_+$.
Then, using Lemma~\ref{comm-lw} and
Proposition~\ref{i-iw-lem}, we get
$$
(I_w)^n=(I^n)_w=\bigg(\bigcap_{i=1}^r\mathfrak{q}_i^n\bigg)_w=
\bigcap_{i=1}^r(\mathfrak{q}_i^n)_w=\bigcap_{i=1}^r((\mathfrak{q}_i)_w)^n.
$$
\quad $\Leftarrow$) Assume that
$(I_w)^n=\bigcap_{i=1}^r((\mathfrak{q}_i)_w)^n$. 
By Lemma~\ref{comm-lw} and Proposition~\ref{i-iw-lem}, we get
$$
(I^n)_w=\bigg(\bigcap_{i=1}^r\mathfrak{q}_i^n\bigg)_w.
$$
\quad Thus, by Lemma~\ref{iw=jw}, we get
$I^n=\bigcap_{i=1}^r\mathfrak{q}_i^n$.
\end{proof}

\begin{corollary}\label{i-ip-coro} If $w$ is a standard weighting 
and $I=\bigcap_{i=1}^r\mathfrak{q}_i$ is a minimal
primary decomposition of a monomial ideal $I$, then
$I_w=\bigcap_{i=1}^r(\mathfrak{q}_i)_w$ is a minimal
primary decomposition of $I_w$.
\end{corollary}

\begin{proof} This follows from Propositions~\ref{stanprim-iff} and
\ref{i-iw-lem}, or from the proof of
\cite[Lemma~3.9]{nasernejad-w}. 
\end{proof}

\begin{proposition}\cite[Theorem~2.6]{Al-Ayyoub-Nasernejad-cover-ideals}
\label{jan24-24} Let $I$ be a monomial ideal of $S$. If $I_w$ is normal and $w$ is
a standard linear weighting, then $I$ is normal. 
\end{proposition}
\begin{proof} To show the equality $I^n=\overline{I^n}$ for all 
$n\geq 1$ we need only show the inclusion $\overline{I^n}\subset I^n$
because the reverse inclusion is clear. Take $t^a\in\overline{I^n}$.
Then, by Lemma~\ref{icd}, $(t^a)^k\in 
(I^{n})^k$ for some $k\geq 1$, and we can write $t^{ka}=t^\delta
t^{c_1}\cdots t^{c_{nk}}$, where $t^\delta\in S$ and $t^{c_i}\in I$
for all $i$. Applying $w$ to $ka$ gives
$$
(t^{w(a)})^k=t^{kw(a)}=t^{w(\delta)}t^{w(c_1)}\cdots t^{w(c_{nk})}\in
((I_w)^{n})^k.
$$
\quad Thus, $t^{w(a)}\in\overline{(I_w)^{n}}=(I_w)^{n}$, and we can
write $t^{w(a)}=t^\gamma t^{w(b_1)}\cdots t^{w(b_n)}$, where $t^\gamma\in
S$ and $t^{b_i}\in I$ for all $i$.  As $w$ is a standard weighting, it is
bijective, and we can write $\gamma=w(\epsilon)$ for some
$\epsilon\in\mathbb{Q}_+^s$. Hence, $a=\epsilon+b_1+\cdots+b_n$, 
$\epsilon\in\mathbb{N}^s$, $t^a\in I^n$, and the proof is complete.  
\end{proof}

\section{Symbolic powers of monomial
ideals}\label{symbolic-monomial-section}
In this section, we study Simis and
normally torsion-free monomial
ideals, and relate some of the properties of a monomial ideal and its
weighted monomial ideal. 

\begin{proposition}\label{jan20-24} Let $I\subset S$ be a monomial ideal with 
$\mathcal{G}(I)\subset K[V]$ for some $V\subset\{t_1,\ldots,t_s\}$ and
let $\mathfrak{N}=(V)$ be the ideal of $S$ generated by $V$. The
following hold.
\begin{enumerate} 
\item[(a)] For each $n\geq 1$, all the
 associated primes of $I^n$ are contained in
$\mathfrak{N}$.
\item[(b)] $I^nS_\mathfrak{N}\cap S=(IS_\mathfrak{N})^n\cap S=I^n$
for all $n\geq 1$. 
\end{enumerate} 
\end{proposition}

\begin{proof}  Since $\mathcal{G}(I^n)\subset K[V]$ for all $n\geq 1$
and recalling that powers of ideals
commute with localizations \cite[p.~76]{monalg-rev}, to show (a) and (b) 
we need only show the case $n=1$.  

(a) Let $\mathfrak{p}$ be an associated prime of $I$, that is, 
$\mathfrak{p}=(I\colon t^b)$ for some $t^b\in S$. 
Take $t_\ell\in\mathfrak{p}$. Then, $t_\ell t^b\in I$ and 
consequently $t_\ell t^b=t^c t^a$, $t^c\in S$, $t^a\in\mathcal{G}(I)$.
We claim that $t_\ell\in V$. If $t_\ell\notin V$, then $t_\ell$ divides $t^c$ because
$t^a\in\mathcal{G}(I)\subset K[V]$. Thus, $t^b\in I$ and
$1\in\mathfrak{p}$, a contradiction.
Hence, as $\mathfrak{p}$ is generated by a subset of
$\{t_1,\ldots,t_s\}$, we get $\mathfrak{p}\subset(V)=\mathfrak{N}$.

(b) This follows from part (a) and Proposition~\ref{jan16-24}. 
\end{proof}

We recover the following two results: 

\begin{corollary}\cite[Proposition~3.6]{cooper-etal}\label{dec17-23-1}
Let $I\subset S$ be a monomial ideal and let $I=\bigcap_{i=1}^r\mathfrak{q}_i$
be a minimal primary decomposition of $I$. If $\mathfrak{p}\in {\rm
Ass}(I)$, then 
$$
I^nS_\mathfrak{p}\cap S=(IS_\mathfrak{p}\cap
S)^n=(\mathfrak{q}_{\subset\mathfrak{p}})^n=(\mathfrak{q}_{\subset\mathfrak{p}})^{\langle
n\rangle}\ \mbox{ for all }\ 
n\geq 1.
$$
\end{corollary}

\begin{proof} Note that, by Proposition~\ref{dec17-23-coro},
$\bigcap_{\mathfrak{p}_i\subset\mathfrak{p}} 
\mathfrak{q}_i$ is a minimal primary decomposition of 
$\mathfrak{q}_{\subset\mathfrak{p}}$, where $\mathfrak{p}_i$ is the
radical of $\mathfrak{q}_i$ for all $i$, and ${\rm
maxAss}(\mathfrak{q}_{\subset\mathfrak{p}})=\{\mathfrak{p}\}$. By
Proposition~\ref{stanprim-iff},
$\mathcal{G}(\mathfrak{q}_i)\subset K[\mathcal{G}(\mathfrak{p}_i)]$
for $i=1,\ldots,r$. 
Then, by Lemma~\ref{intofmonoismono}, one has
$$\mathcal{G}(\mathfrak{q}_{\subset\mathfrak{p}})\subset
K[\cup_{\mathfrak{p}_i\subset\mathfrak{p}}\mathcal{G}(\mathfrak{p}_i)]=
K[\mathcal{G}(\mathfrak{p})].
$$
\quad Making $V=\mathcal{G}(\mathfrak{p})$ and
$\mathfrak{N}=(V)=\mathfrak{p}$, by Proposition~\ref{jan20-24}(a), all
associated primes of $(\mathfrak{q}_{\subset\mathfrak{p}})^n$ are
contained in $\mathfrak{p}$ and, by the equality ${\rm
maxAss}(\mathfrak{q}_{\subset\mathfrak{p}})=\{\mathfrak{p}\}$ 
and Proposition~\ref{jan20-24}(b), one has
$$(\mathfrak{q}_{\subset\mathfrak{p}})^{\langle
n\rangle}=((\mathfrak{q}_{\subset\mathfrak{p}})^n)S_\mathfrak{p}\cap S=
((\mathfrak{q}_{\subset\mathfrak{p}})S_\mathfrak{p})^n\cap S=
(\mathfrak{q}_{\subset\mathfrak{p}})^n\mbox{ for all }n\geq 1.
$$
\quad Hence, since powers of ideals commute with localizations, 
by Proposition~\ref{dec17-23-coro}, we get 
\begin{align*}
I^nS_\mathfrak{p}\cap S=&(IS_\mathfrak{p})^n\cap
S=((\mathfrak{q}_{\subset\mathfrak{p}})S_\mathfrak{p})^n\cap S
=((\mathfrak{q}_{\subset\mathfrak{p}})^n)S_\mathfrak{p}\cap S=
(\mathfrak{q}_{\subset\mathfrak{p}})^n
\end{align*}
and $\mathfrak{q}_{\subset\mathfrak{p}}=IS_\mathfrak{p}\cap S$. 
\end{proof}

\begin{corollary}\cite[Theorem~3.7]{cooper-etal}\label{dec17-23-2}
Let $I\subset S$ be a monomial ideal and let $I=\bigcap_{i=1}^r\mathfrak{q}_i$
be a minimal primary decomposition of $I$ with ${\rm
rad}(\mathfrak{q}_i)=\mathfrak{p}_i$. Then 
$$
I^{\langle n\rangle}:=\bigcap_{i=1}^r(I^nS_{\mathfrak{p}_i}\cap S)=
\bigcap_{\mathfrak{p}_i\in {\rm maxAss}(I)}\hspace{-4mm}(I^nS_{\mathfrak{p}_i}\cap
S)=\bigcap_{\mathfrak{p}_i\in {\rm
maxAss}(I)}\hspace{-4mm}(\mathfrak{q}_{\subset\mathfrak{p}_i})^n\ \mbox{ for all }\ 
n\geq 1.
$$
\end{corollary}

\begin{proposition}\label{ntf-char-coro} 
Let $I$ be a monomial ideal of $S$. Then, $I^n=I^{(n)}$ for all $n\geq 1$ if and only 
if $I$ has no embedded primes and ${\rm Ass}(I^n)\subset{\rm Ass}(I)$ for all $n\geq 1$.
\end{proposition}

\begin{proof} Let $I=\bigcap_{i=1}^r\mathfrak{q}_i$ be
a minimal primary decomposition of $I$ with ${\rm
rad}(\mathfrak{q}_i)=\mathfrak{p}_i$ for all $i$. 

$\Rightarrow$) Since
$I=I^{(1)}$, 
$I$ has no
embedded primes because ${\rm Ass}(I)$ is independent
of the minimal primary decomposition of $I$ that we choose \cite{AM}. By
Proposition~\ref{stanprim-iff}, $\mathfrak{q}_i^n$ is
$\mathfrak{p}_i$-primary for all $i$, and by
Lemma~\ref{anoth-one-char-spow-general}, we get the following primary
decomposition of $I^n$:
$$  
I^n=I^{(n)}=\mathfrak{q}_1^n\cap\cdots\cap\mathfrak{q}_r^n.
$$
\quad Hence, ${\rm Ass}(I^n)\subset{\rm Ass}(I)$ for all $n\geq 1$. 

$\Leftarrow$) As $I$ has no embedded primes, one has 
$$
\{\mathfrak{p}_1,\ldots,\mathfrak{p}_r\}={\rm Ass}(I)=
{\rm minAss}(I)={\rm Min}(I)={\rm Min}(I^n)\subset{\rm
Ass}(I^n)\subset{\rm Ass}(I)
$$
for all $n\geq 1$. Hence, $I^n$ has a unique minimal primary decomposition
$I^n=\bigcap_{i=1}^rQ_i$ with ${\rm rad}(Q_i)=\mathfrak{p}_i$ for
$i=1,\ldots,r$. Therefore, by Proposition \ref{dec17-23-coro}, we get
$$
I^{(n)}=\bigcap_{i=1}^r(I^nS_{\mathfrak{p}_i}\cap
S)=\bigcap_{i=1}^r(Q_iS_{\mathfrak{p}_i}\cap
S)=\bigcap_{i=1}^rQ_i=I^n,
$$
and the proof is complete.
\end{proof}

\begin{corollary}\label{simis-w} Let $w$ be a standard linear
weighting, let $I$
be a monomial ideal of $S$, and let $I_w$ be its weighted monomial
ideal. The following hold.
\begin{enumerate}
\item[(a)] If $I$ has no embedded primes, 
then $I$ is a Simis ideal in degree $n$ if and only if $I_w$ is a
Simis ideal in degree $n$.
\item[(b)] {\rm(cf. \cite[Lemma~1]{jtt1})} $I$ is a Simis ideal if and only if $I_w$ is a Simis ideal.
\item[(c)] ${\rm Ass}(I^n)={\rm Ass}((I_w)^n)$ for all $n\geq 1$. 
\item[(d)] \cite[Theorem~3.10]{nasernejad-w} $I$ is normally torsion-free if and only if $I_w$ is
normally torsion-free.
\item[(e)] \cite[Lemma~4]{cm-oriented-trees} 
$I$ is Cohen--Macaulay if and only if $I_w$ is
Cohen--Macaulay.
\item[(f)] {\rm(\cite[p.~536]{oriented-graphs}, 
\cite[Lemma~3.9]{nasernejad-w})} $I$ is unmixed if and only if $I_w$ is
unmixed.
\end{enumerate}
\end{corollary}

\begin{proof} Let $I=\bigcap_{i=1}^r\mathfrak{q}_i$ be the minimal primary
decomposition of $I$ in Eq.~\eqref{oct28-23} with ${\rm
rad}(\mathfrak{q}_i)=\mathfrak{p}_i$ for $i=1,\ldots,r$. Noticing that
$(\mathfrak{q}_i)_w$ is $\mathfrak{p}_i$-primary, by
Proposition~\ref{i-iw-lem} and Lemma~\ref{iw=jw}, we get
\begin{equation}\label{feb4-24}
I_w=(\mathfrak{q}_1)_w\cap\cdots\cap(\mathfrak{q}_r)_w
\end{equation}
and this is a minimal primary decomposition of $I_w$. 

(a) $\Rightarrow$) Assume that $I$ is Simis in degree $n$, that is, 
$I^n=I^{(n)}$. As $I$ has no embedded primes, 
by Eq.~\eqref{feb4-24}, Lemma~\ref{anoth-one-char-spow-general} and Theorem~\ref{i-iw-thm},
we obtain
$$
(I_{w})^{(n)}=((\mathfrak{q}_1)_w)^n\cap\cdots\cap((\mathfrak{q}_r)_w)^n=(I_w)^n,
$$
that is, $I_w$ is Simis in degree $n$ and the proof is complete. 

$\Leftarrow$) Assume that $I_w$ is Simis in degree $n$, that is, 
$(I_w)^{(n)}=(I_w)^n$. 
As $I$ has no embedded primes, by Eq.~\eqref{feb4-24} and 
Lemma~\ref{anoth-one-char-spow-general}, we get 
$$
(I_w)^n=(I_w)^{(n)}=((\mathfrak{q}_1)_w)^n\cap\cdots\cap((\mathfrak{q}_r)_w)^n,
$$
and by Lemma~\ref{anoth-one-char-spow-general} and 
Theorem~\ref{i-iw-thm}, we get $I^n=I^{(n)}$, that is, $I$ is Simis in
degree $n$.

(b) $\Rightarrow$) Assume that $I^n=I^{(n)}$ for all $n\geq 1$. Since
$I=I^{(1)}$, $I$ has no embedded primes. Then, by part (a), 
we obtain $(I_{w})^{(n)}=(I_w)^n$
for all $n\geq 1$ and the proof is complete. 

$\Leftarrow$) Assume that $(I_w)^{(n)}=(I_w)^n$ for all $n\geq 1$. 
As $(I_w)^{(1)}=I_w$, $I_w$ and $I$ have no embedded primes. Then, by
part (a), $I^n=I^{(n)}$ for all $n\geq 1$. 

(c) Fix $n\in\mathbb{N}_+$. Let $I^n=\bigcap_{i=1}^p Q_i$ be a minimal primary
decomposition of the monomial ideal $I^n$ with ${\rm
rad}(Q_i)=P_i$ for $i=1,\ldots,p$. Then, by
Proposition~\ref{i-iw-lem} and Lemmas~\ref{comm-lw}--\ref{iw=jw}, we get
\begin{equation}
(I_w)^n=(I^n)_w=(Q_1)_w\cap\cdots\cap(Q_p)_w
\end{equation}
and this is a minimal primary decomposition of $(I_w)^n$. 
Hence, noticing that
$(Q_i)_w$ is $P_i$-primary for all $i$, we get ${\rm Ass}(I^n)=
\{P_1,\ldots,P_p\}={\rm Ass}((I_w)^n)$.

(d) and (f) follow from part (c). 
\end{proof}

The following result relates the notions of normally torsion-free ideals
and Simis ideals using the definition of symbolic powers in terms of 
the set of all associated primes.

\begin{proposition}\label{ntf-implies-simis-all-ass} Let $I\subset S$ be an ideal. If ${\rm
Ass}(I^n)\subset{\rm Ass}(I)$ for some $n\geq 1$, then 
$I^{\langle n\rangle}=I^n$. In particular, if $I$ is normally
torsion-free, then $I^{\langle n\rangle}=I^n$ for all $n\geq 1$.
\end{proposition}

\begin{proof} We may assume ${\rm
Ass}(I^n)=\{\mathfrak{p}_1,\ldots,\mathfrak{p}_{r_1}\}$, 
${\rm Ass}(I)=\{\mathfrak{p}_1,\ldots,\mathfrak{p}_{r}\}$ and $r_1\leq r$. Let
$$I^n=\mathfrak{q}_1'\cap\cdots\cap\mathfrak{q}_{r_1}'$$
be a minimal
primary decomposition of $I^n$ with $\mathfrak{q}_i'$ a
$\mathfrak{p}_i$-primary ideal for $i=1,\ldots,r_1$. 
 Since $I^{\langle n\rangle}\supset I^n$, it suffices to
show the inclusion $I^{\langle n\rangle}\subset I^n$. Take $f\in
I^{\langle n\rangle}$, that is, $f\in I^nS_{\mathfrak{p}_i}\textstyle\cap
S$ for $i=1,\ldots,r$. For each $1\leq i\leq r$, we can write
$f=f_i/g_i$, $f_i\in I^n$, $g_i\notin\mathfrak{p}_i$. Thus, 
$fg_i=f_i\in I^n$. If $f\notin\mathfrak{q}_j'$ for some $1\leq j\leq
r_1\leq r$, then $g_jf=f_j\in I^n$, $g_jf\in\mathfrak{q}_j'$, 
and $g_j^m\in\mathfrak{q}_j'$ for some $m\geq 1$. Thus, 
$g_j\in{\rm rad}(\mathfrak{q}_j')=\mathfrak{p}_j$, a contradiction.
Then, $f\in\mathfrak{q}_j'$ for all $1\leq j\leq r_1$, and $f\in I^n$.
\end{proof}

\begin{conjecture}\label{conjecture-simis} Let $I$ be a monomial ideal of $S$ without embedded
primes. If the irreducible decomposition of $I$ is minimal and $I$ is
a Simis ideal, then there is a Simis squarefree monomial ideal $J$
of $S$ and a standard linear weighting $w$ such that $I=J_w$.
\end{conjecture}

The following results give some support for this conjecture and shows
how to construct non-squarefree Simis monomial ideals. 

\begin{corollary}{\rm(\cite[Corollary~4.8]{Mandal-Pradhan1},
\cite{ITG})}\label{char-ntf-w}
Let $w$ be a standard linear weighting of $S$, let $G$ be a graph with
vertex set $\{t_1,\ldots,t_s\}$, and let $I=I(G)$ be the edge ideal
of $G$. The following conditions are equivalent.
\begin{enumerate}
\item[(a)] $I^{(n)}=I^n$ for all $n\geq 1$;
\item[(b)] $(I_w)^{(n)}=(I_w)^n$ for all $n\geq 1$;
\item[(c)] $G$ is a bipartite graph.
\end{enumerate}
\end{corollary}

\begin{proof} The equivalence of (a) and (c) is a classical
result \cite[Theorem~5.9]{ITG}. As $I(G)$ has no embedded primes, the
equivalence of (a) and (b) follows 
at once from Corollary~\ref{simis-w}(a).
\end{proof}

\begin{proposition}\label{jan25-24}
If $I\subset S$ is a Simis monomial ideal and $k\in\mathbb{N}_+$, then
$I^k$ is a Simis ideal.
\end{proposition}

\begin{proof} Let 
$\mathfrak{q}_1,\ldots,\mathfrak{q}_p$ be the primary components
of $I$ corresponding to the  minimal primes 
of $I$ and let $\mathfrak{p}_i$ be the radical of $\mathfrak{q}_i$. Setting $J=I^k$, by
Lemma~\ref{anoth-one-char-spow-general}, one has
$$
J=I^k=I^{(k)}=\mathfrak{q}_1^k\cap\cdots\cap\mathfrak{q}_p^k.
$$
\quad As $\mathfrak{q}_i^k$ is again a $\mathfrak{p}_i$-primary ideal for
$i=1,\ldots,p$ (this follows from Proposition~\ref{stanprim-iff}), 
$\mathfrak{q}_1^k,\ldots,\mathfrak{q}_p^k$ are the primary components
of $J$ that correspond to the minimal primes of $J$. Hence, again by 
Lemma~\ref{anoth-one-char-spow-general}, we have
\begin{align*}
J^{(n)}=&(\mathfrak{q}_1^k)^n\cap\cdots\cap(\mathfrak{q}_p^k)^n=
\mathfrak{q}_1^{kn}\cap\cdots\cap \mathfrak{q}_p^{kn}=I^{(kn)}=I^{kn}=
(I^k)^n=J^n,
\end{align*} 
for all $n\geq 1$, and the proof is complete.
\end{proof}

\section{Symbolic powers of squarefree monomial
ideals}\label{symbolic-squarefree}
In this section, we give a structure theorem for
edge ideals of $d$-uniform clutters whose ideal of covers is Simis in
degree $d$ and give another
algebraic classification of bipartite graphs using the 2nd symbolic power of
ideals of covers of graphs.  To avoid
repetitions, we continue to employ 
the notations and definitions used in Section~\ref{intro-section}.

A vertex $t_i$ of a clutter $\mathcal{C}$ is called \textit{isolated}
if $t_i$ is not in any edge $e$ of $\mathcal{C}$.

\begin{theorem}\label{dual-d-simis}
Let $\mathcal{C}$ be a $d$-uniform clutter with vertex set
$V(\mathcal{C})=\{t_1,\ldots,t_s\}$ and let $I_c(\mathcal{C})$ be
its ideal of covers. If
$I_c(\mathcal{C})^{(d)}=I_c(\mathcal{C})^{d}$ and $\mathcal{C}$ has no
 isolated vertices, then there are 
mutually disjoint minimal vertex covers $C_1,\ldots,C_d$ of
$\mathcal{C}$ such that $V( \mathcal{C})=\bigcup_{i=1}^d C_i$ and
every edge of $\mathcal{C}$ has the form 
$e=\{t_{i_1},\ldots,t_{i_d}\}$, where $t_{i_j}\in C_j$ for all $j$.
\end{theorem}

\begin{proof} We set $f=t_1t_2\cdots t_s$. Let
$e=\{t_{i_1},\ldots,t_{i_d}\}$ be any edge of $\mathcal{C}$. Then,
$t_{i_1}\cdots t_{i_d}\in(e)^d$, and consequently 
$f\in\bigcap_{e\in E(\mathcal{C})}(e)^d$. Since 
$I_c(\mathcal{C})=\bigcap_{e\in E(\mathcal{C})}(e)$, by 
Lemma~\ref{anoth-one-char-spow-general}, one has
$$
\bigcap_{e\in
E(\mathcal{C})}(e)^d=I_c(\mathcal{C})^{(d)}=I_c(\mathcal{C})^{d}.
$$ 
\quad Thus, $f\in I_c(\mathcal{C})^{d}$. Hence, we can write 
$f=t^\delta t_{C_1}\cdots t_{C_d}$, 
where $C_1,\ldots,C_d$ are minimal vertex covers of $\mathcal{C}$ and 
$t_{C_i}=\prod_{t_j\in C_i}t_j$ for all $i$. Since $f$ is squarefree, 
$C_1,\ldots,C_d$ are mutually disjoint. To show the equality $V(
\mathcal{C})=\bigcup_{i=1}^d C_i$ we argue by contradiction assuming 
$\bigcup_{i=1}^d C_i
\subsetneq V(\mathcal{C})$. Take $t_i\in V(\mathcal{C})\setminus
\bigcup_{i=1}^d C_i$. Since $\mathcal{C}$ has no isolated vertices, we
can pick an edge $e$ containing $t_i$. Since $e\cap C_i\neq\emptyset$
for $i=1,\ldots,d$ and $C_k \cap C_j=\emptyset$ for $k\neq j$, we
obtain $|e|\geq d+1$, a contradiction because all edges of
$\mathcal{C}$ contain exactly $d$-elements. Then, 
$V(\mathcal{C})=\bigcup_{i=1}^d C_i$. Now take any edge $e$ of
$\mathcal{C}$, then $|e|=d$ and $e\cap C_i\neq\emptyset$. Thus, 
$e=\{t_{i_1},\ldots,t_{i_d}\}$, where $t_{i_j}\in C_j$ for all $j$.
\end{proof}

\begin{proposition}\label{dec11-23}
Let $G$ be a graph without isolated vertices and 
let $I_c(G)$ be the ideal of covers of $G$. The following 
conditions are equivalent:
\begin{enumerate}
\item[(i)] $I_c(G)^{(2)}=I_c(G)^2$.
\item[(ii)] $G$ is a bipartite graph.
\end{enumerate}
\end{proposition}

\begin{proof} (i) $\Rightarrow$ (ii)  The graph $G$ is a $2$-uniform
clutter and $I_c(G)^{(2)}$ is equal to $I_c(G)^2$. Then, by
Theorem~\ref{dual-d-simis}, there are minimal vertex covers 
$V_1,V_2$ of $G$ such that $V_1\cap V_2=\emptyset$, 
$V(G)=\bigcup_{i=1}^2 V_i$, and
every edge of $G$ has the form 
$e=\{t_{i_1},t_{i_2}\}$, where $t_{i_j}\in V_j$ for $j=1,2$. This
means that $(V_1,V_2)$ is a bipartition of the graph $G$, that is, 
$G$ is bipartite.

 (ii) $\Rightarrow$ (i) By \cite[Corollary~4.6]{alexdual}, one has
 $I_c(G)^{(n)}=I_c(G)^n$ for all $n\geq 1$. 
\end{proof}

\section{Symbolic powers of the dual of edge ideals of oriented
graphs}\label{symbolic-dual-section}

To avoid 
repetitions, we continue to employ 
the notations and definitions used in Sections~\ref{intro-section} and
\ref{prelim-section}.
In this section, we classify combinatorially 
when the dual of
the edge ideal of a weighted oriented graph is a Simis ideal in degree
$2$ and when it is a Simis ideal. 

\begin{theorem}\cite[Corollary~3.17, Theorem~4.6,
Proposition~4.27]{reesclu}\label{ntf-dual} If $G$ is a graph, then $J(G)$ is Simis if and
only if $G$ is bipartite. 
\end{theorem}

\begin{theorem}\cite[Theorem~3.2]{weighted-symbolic}\label{I2=I(2)-char}  
Let $D$ be a weighted oriented graph and let $G$ be its underlying 
graph. Then, $I(D)^2=I(D)^{(2)}$ if and only if the following
two conditions hold:
\begin{enumerate}
\item[{\rm (i)}] Every vertex in $V^+(D)$ is a sink;
\item[{\rm (ii)}] $G$ has no triangles.
\end{enumerate}
\end{theorem}

The following lemma is the main auxiliary result of this section.

\begin{lemma}\label{jan4-24}
Let $D$ be a weighted oriented graph, let $I(D)$ be its edge ideal, 
and let $J(D)$ be the dual of
$I(D)$. If there exists
a vertex $v\in V^+(D)$ that is neither a source nor a sink, then
$J(D)^{(2)}\not\subset J(D)^2$. 
\end{lemma}

\begin{proof} There are $u,x$ in $V(D)$ such that $(u,v)$, $(v,x)$
are in $E(D)$. Let $V(D)=\{t_1,\ldots,t_s\}$ be the vertex 
set of $D$ and let $w_i$ be the weight of $t_i$. By
Lemma~\ref{anoth-one-char-spow-general}, one has
\begin{equation}\label{dec28-23}
J(D)^{(2)}=\bigcap_{(t_i,t_j)\in E(D)}
(t_i,t_j^{w_j})^2.
\end{equation}
\quad We may assume
$u=t_1$, $v=t_2$, $x=t_3$, and $w_2\geq 2$. Then, the monomial ideals 
$J_1:=(t_1,t_2^{w_2})$ and
$J_2:=(t_2,t_3^{w_3})$ are irreducible components of $J(D)$. 
Take $(t_i,t_j)\in E(D)$, $i\neq j$, and consider the ideal
$I_{i,j}=(t_i,t_j^{w_j})^2=(t_i^2,t_it_j^{w_j},t_j^{2w_j})$. There
are two 
cases to consider.

\begin{enumerate}
\item[(A)] Assume that $t_1,t_2,t_3$ do not form a
triangle of the underlying graph $G$ of $D$, that is, $(t_1,t_3)$ and
$(t_3,t_1)$ are not edges of $D$. Setting
$f=t_1t_2^{w_2}t_4^{2w_4}\cdots t_s^{2w_s}$, it suffices to show 
that $f\in J(D)^{(2)}\setminus J(D)^2$. Using Eq.~\eqref{dec28-23}, 
first we show that $f\in J(D)^{(2)}$, that is, we show that $f\in
I_{i,j}^2$ for all $(t_i,t_j)\in E(D)$.  
We consider three 
subcases. 
\begin{enumerate}
\item[(A.1)] Assume that $i=1$. If $j=2$, then $t_1t_2^{w_2}\in
I_{1,2}^2$, $t_1t_2^{w_2}$ divides $f$, and
$f\in I_{1,2}^2$. The case $j=3$ cannot occur because $t_1,t_2,t_3$ do not form a
triangle of $G$. If $j\geq 4$, then $t_j^{2w_j}\in I_{1,j}^2$,
$t_j^{2w_j}$ divides $f$, and
$f\in I_{1,j}^2$.

\item[(A.2)] Assume that $i\geq 2$ and $j=1$. Note that
$(t_2,t_1)\notin E(D)$ and $(t_3,t_1)\notin E(D)$ because $D$ is an
oriented graph and $t_1,t_2,t_3$ do not form a triangle of $G$. Then,
$i\geq 4$, $t_i^2\in I_{i,1}^2$, $t_i^2$ divides $t_i^{2w_i}$, and $t_i^{2w_i}$ divides $f$.
Thus, $t_i^2$ divides $f$ and $f\in I_{i,1}^2$.

\item[(A.3)] Assume that $i\geq 2$ and $j\geq 2$. If $i=2$, then
$j\geq 3$, $t_2^2\in I_{2,j}^2$, $t_2^2$ divides $f$ because $w_2\geq 2$, and $f\in
I_{2,j}^2$. If $i=3$, then $j\geq 4$ because $(t_3,t_2)\notin E(D)$, 
$t_j^{2w_j}\in I_{3,j}^2$, and  $t_j^{2w_j}$ divides $f$. Thus,
$f\in I_{3,j}^2$.  If $i\geq 4$, then $t_i^2\in I_{i,j}^2$, $t_i^2$
divides $t_i^{2w_i}$,  
and $t_i^{2w_i}$ divides $f$. Thus, one has $f\in I_{i,j}^2$. 
\end{enumerate}
\quad Hence $f\in J(D)^{(2)}$. Now we show that $f\notin J(D)^2$. By 
Lemma~\ref{intofmonoismono}, one has 
\begin{align*}
J(D)\subset&(t_1,t_2^{w_2})\cap(t_2,t_3^{w_3})=(t_1t_2,t_1t_3^{w_3},t_2^{w_2}),\\
J(D)^2\subset&(t_1t_2,t_1t_3^{w_3},t_2^{w_2})^2=
(t_1^2t_2^2,t_1^2t_2t_3^{w_3},t_1t_2^{w_2+1},t_1^2t_3^{2w_3},
t_1t_2^{w_2}t_3^{w_3},t_2^{2w_2}).
\end{align*}
Therefore, $f\notin(t_1t_2,t_1t_3^{w_3},t_2^{w_2})^2$, and
consequently $f\notin J(D)^2$. 

\item[(B)] Assume that $t_1,t_2,t_3$ form a triangle of the underlying
graph $G$ of $D$, that is, either $(t_1,t_3)\in E(D)$ or $(t_3,t_1)\in
E(D)$. There are two cases to consider.

\begin{enumerate}
\item[(B.1)] Assume that $(t_1,t_3)\in E(D)$. Setting 
$g=t_1t_2^{w_2}t_3^{w_3}t_4^{2w_4}\cdots t_s^{2w_s}$, it suffices to show 
that $g\in J(D)^{(2)}\setminus J(D)^2$. Using Eq.~\eqref{dec28-23}, 
first we show that $g\in J(D)^{(2)}$, that is, we show that $g\in
I_{i,j}^2$ for all $(t_i,t_j)\in E(D)$.
\begin{enumerate}
\item[(B.1.1)] Assume that $i=1$. Then, $j\geq 2$, $t_1t_j^{w_j}$
divides $g$ for $j=2,3$, and $t_j^{2w_j}$ divides $g$ for $j\geq 4$. 
Thus, $g\in  I_{1,j}^2$ for $j\geq 2$. 
\item[(B.1.2)] Assume that $j=1$. Then, $i\geq 4$ because
$(t_1,t_2)\in E(D)$, $(t_1,t_3)\in E(D)$ and $D$ is oriented. Hence,
$t_i^2$ divides $g$ and $g\in I_{i,1}^2$.
\item[(B.1.3)] Assume that $i\geq 2$, $j\geq 2$. Then, for $j\geq 4$, 
$t_j^{2w_j}$ divides $g$, and for $i\geq 4$, $t_i^2$ divides $g$. If
$2\leq i\leq 3$ and $2\leq j\leq 3$, then $i=2$ and $j=3$ because 
$(t_3,t_2)\notin E(D)$. Since $w_2\geq 2$, $t_2^2$ divides $g$. Thus,
in each case $g\in I_{i,j}^2$.
\end{enumerate}
\quad Hence $g\in J(D)^{(2)}$. Now we show that $g\notin J(D)^2$. By 
Lemma~\ref{intofmonoismono}, one has 
\begin{align*}
J(D)\subset&(t_1,t_2^{w_2})\cap(t_2,t_3^{w_3})\cap(t_1,t_3^{w_3})=
(t_1t_2,t_1t_3^{w_3},t_2^{w_2}t_3^{w_3}),\\
J(D)^2\subset&(t_1^2t_2^2,t_1^2t_2t_3^{w_3},t_1t_2^{w_2+1}t_3^{w_3},t_1^2t_3^{2w_3},
t_1t_2^{w_2}t_3^{2w_3},t_2^{2w_2}t_3^{2w_3}).
\end{align*}
Therefore, $g\notin(t_1t_2,t_1t_3^{w_3},t_2^{w_2}t_3^{w_3})^2$, and
consequently $g\notin J(D)^2$. 
\item[(B.2)] Assume that $(t_3,t_1)\in E(D)$. There are three cases to
consider.
\begin{enumerate}
\item[(B.2.1)] Assume that $w_1=1$, $w_3=1$. Note that this 
case follows from case (B.1) since we can reverse the direction of
arrow $(t_3,t_1)$ without changing $J(D)$. 
\item[(B.2.2)] Assume that $w_1=1$, $w_3\geq 2$. Setting 
$g=t_1t_2^{w_2}t_3^{w_3}t_4^{2w_4}\cdots t_s^{2w_s}$, it suffices to show 
that $g\in J(D)^{(2)}\setminus J(D)^2$. Using Eq.~\eqref{dec28-23}, 
first we show that $g\in J(D)^{(2)}$, that is, we show that $g\in
I_{i,j}^2$ for all $(t_i,t_j)\in E(D)$.  
We consider three 
subcases. 
\begin{enumerate}
\item[(B.2.2.1)] Assume that $i=1$. Then, $j\geq 2$. 
If $j=2,3$, then $t_1t_j^{w_j}$ divides $g$ and
$g\in I_{1,j}^2$. If $j\geq 4$, then $t_j^{2w_j}$ divides $g$ and
$g\in I_{1,j}^2$.

\item[(B.2.2.2)] Assume that $j=1$. Then, $i\geq 3$ because $D$ is
oriented. If $i=3$, then $t_3^2$ divides $g$ because $w_3\geq 2$ and
$g\in I_{3,1}^2$. If 
$i\geq 4$, then $t_i^2$ divides $t_i^{2w_i}$, $t_i^{2w_i}$ divides
$g$, and $g\in I_{i,1}^2$.

\item[(B.2.2.3)] Assume that $i\geq 2$ and $j\geq 2$. If $i\geq 4$,
then $t_i^2$ divides $g$ and $g\in I_{i,j}^2$. If $j\geq 4$, then
$t_j^{2w_j}$ divides $g$ and $g\in I_{i,j}^2$. If $2\leq i\leq 3$ and
$2\leq j\leq 3$, then $i=2$, $j=3$ because $D$ is oriented. Then
$t_2^2$ divides $g$ because $w_2\geq 2$ and $g\in I_{2,3}^2$. 
\end{enumerate}
\quad Hence $g\in J(D)^{(2)}$. Now we show that $g\notin J(D)^2$. By 
Lemma~\ref{intofmonoismono}, one has 
\begin{align*}
J(D)\subset&(t_1,t_2^{w_2})\cap(t_2,t_3^{w_3})\cap(t_3,t_1^{w_1})=
(t_1t_2t_3,t_1^{w_1}t_2,t_1t_3^{w_3},t_2^{w_2}t_3),\\
J(D)^2\subset&(t_1^2t_2^2t_3^2,t_1^{w_1+1}t_2^2t_3,t_1^2t_2t_3^{w_3+1},
t_1t_2^{w_2+1}t_3^{2},
t_1^{2w_1}t_2^{2},t_1^{w_1+1}t_2t_3^{w_3},\\
\quad &t_1^{w_1}t_2^{w_2+1}t_3,t_1^2t_3^{2w_3},
t_1t_2^{w_2}t_3^{w_3+1},t_2^{2w_2}t_3^2).
\end{align*}
Therefore, $g\notin(t_1t_2t_3,t_1^{w_1}t_2,t_1t_3^{w_3},t_2^{w_2}t_3)^2$, and
consequently $g\notin J(D)^2$. 
\item[(B.2.3)] Assume that $w_1\geq 2$, $w_3\geq 1$. Setting 
$h=t_1^{w_1}t_2t_3^{w_3}t_4^{2w_4}\cdots t_s^{2w_s}$, it suffices to show 
that $h\in J(D)^{(2)}\setminus J(D)^2$. Using Eq.~\eqref{dec28-23}, 
first we show that $h\in J(D)^{(2)}$, that is, we show that $h\in
I_{i,j}^2$ for all $(t_i,t_j)\in E(D)$. There are three subcases to consider. 
\begin{enumerate}
\item[(B.2.3.1)] Assume that $i=1$. Then, $j\geq 2$, $t_1^2$ divides
$h$ because $w_1\geq 2$, and $h\in I_{1,j}^2$. 
\item[(B.2.3.2)] Assume that $j=1$. As $D$ is an oriented graph and
$(t_1,t_2)\in E(D)$, then $(t_2,t_1)\notin E(D)$. Thus, $i\geq 3$. If
$i=3$, then $t_1^{w_1}t_3$ divides $h$ and $h\in I_{3,1}^2$. If $i\geq
4$, then $t_i^2$ divides $h$ and $h\in I_{i,1}^2$.
\item[(B.2.3.3)] Assume that $i\geq 2$ and $j\geq 2$. If $i\geq 4$,
then $t_i^2$ divides $h$ and $h\in I_{i,j}^2$. If $j\geq 4$, then
$t_j^{2w_j}$ divides $h$ and $h\in I_{i,j}^2$. If $2\leq i\leq 3$ and
$2\leq j\leq 3$, then $i=2$, $j=3$ because $D$ is oriented. Then, 
$t_2t_3^{w_3}$ divides $h$ and $h\in I_{2,3}^2$.
\end{enumerate}
\quad Hence $h\in J(D)^{(2)}$. Now we show that $h\notin J(D)^2$. By 
Lemma~\ref{intofmonoismono}, one has 
\begin{align*}
J(D)\subset&(t_1,t_2^{w_2})\cap(t_2,t_3^{w_3})\cap(t_3,t_1^{w_1})=
(t_1t_2t_3,\,t_1^{w_1}t_2,\,t_1t_3^{w_3},\,t_2^{w_2}t_3),\\
J(D)^2\subset&(t_1^2t_2^2t_3^2,\,t_1^{w_1+1}t_2^2t_3,\,t_1^2t_2t_3^{w_3+1},\,
t_1t_2^{w_2+1}t_3^{2},\,
t_1^{2w_1}t_2^{2},\,t_1^{w_1+1}t_2t_3^{w_3},\\
\quad&t_1^{w_1}t_2^{w_2+1}t_3,\,t_1^2t_3^{2w_3},\,
t_1t_2^{w_2}t_3^{w_3+1},\,t_2^{2w_2}t_3^2).
\end{align*}
Thus, $h\notin(t_1t_2t_3,t_1^{w_1}t_2,t_1t_3^{w_3},t_2^{w_2}t_3)^2$, and
consequently $h\notin J(D)^2$. 
\end{enumerate}
\end{enumerate}
\end{enumerate}
\quad Therefore, $J(D)^{(2)}\setminus J(D)^2\neq\emptyset$
and the proof is complete.
\end{proof}

We come to one of our main results.

\begin{theorem}\label{J2=J(2)}  
Let $D$ be a weighted oriented graph and let $J(D)$ be the dual of 
the edge ideal $I(D)$ of $D$. Then, $J(D)^2=J(D)^{(2)}$ if and only if the following
two conditions hold:
\begin{enumerate}
\item[{\rm (i)}] Every vertex in $V^+(D)$ is a sink.
\item[{\rm (ii)}] The underlying 
graph $G$ of $D$ is bipartite.
\end{enumerate}
\end{theorem}

\begin{proof} $\Rightarrow$) To prove (i), we argue by contradiction assuming there
is $v$ in $V^+(D)$ which is not a sink. Note that $v$ is not a
source because all sources of $D$ have weight $1$ (Remark~\ref{nov30-21}). Hence, by
Lemma~\ref{jan4-24}, $J(D)^2\neq J(D)^{(2)}$, a contradiction. 

To prove (ii), let $J(G)$ be the ideal of covers of the underlying graph
$G$ of $D$, that is,
$$
J(G)=\bigcap_{\{t_i,t_j\}\in E(G)}
(t_i,t_j).
$$
\quad Consider the standard linear weighting given by
$$
w\colon\mathbb{R}^s\rightarrow \mathbb{R}^s,\quad
a\mapsto(a_1w_1,\ldots,a_sw_s),\ a=(a_1,\ldots,a_s),
$$
\quad Recall that $w(t_i):=w_i$ for $i=1,\ldots,s$. 
We claim that $J(G)_w=J(D)$. Take $\{t_i,t_j\}\in E(G)$ and consider
the following cases:
\begin{enumerate}
\item $w_i=w_j=1$. Then, $t_i^{w_i}t_j^{w_j}=t_it_j$, and 
either $(t_i,t_j)\in E(D)$ or $(t_j,t_i)\in E(D)$. 
\item $w_i\geq 2$, $w_j=1$. Then, $t_i^{w_i}t_j^{w_j}=t_i^{w_i}t_j$, and 
$(t_j,t_i)\in E(D)$ by condition (i). 
\item $w_i=1$, $w_j\geq 2$. Then,
$t_i^{w_i}t_j^{w_j}=t_it_j^{w_j}$, and 
$(t_i,t_j)\in E(D)$ by condition (i). 
\item $w_i\geq 2$, $w_j\geq 2$. Then, $(t_i,t_j)\in E(D)$ and 
$(t_j,t_i)\in E(D)$ because $t_i$ and $t_j$ are sinks by condition
(i), contradicting that $D$ is oriented. Thus, this case cannot occur. 
\end{enumerate}
\quad Therefore, using (1)--(4) and Lemma~\ref{i-iw}, we get 
\begin{equation}\label{feb6-24}
J(G)_w=\bigg(\bigcap_{\{t_i,t_j\}\in E(G)}
(t_i,t_j)\bigg)_w=\bigcap_{\{t_i,t_j\}\in E(G)}
(t_i,t_j)_w=\bigcap_{\{t_i,t_j\}\in E(G)}
(t_i^{w_i},t_j^{w_j})=J(D),
\end{equation}
and $J(G)_w=J(D)$, as claimed. Then, $J(G)_w$ is Simis in degree $2$ and, by
Corollary~\ref{simis-w}(a), so is $J(G)$, that
is, $J(G)^2=J(G)^{(2)}$. Therefore, by
Proposition~\ref{dec11-23}, $G$ is a bipartite graph.  

$\Leftarrow$) As $G$ is bipartite, by Proposition~\ref{dec11-23}, one
has $J(G)^{2}=J(G)^{(2)}$. By condition (i) 
and Lemma~\ref{i-iw}, we get $J(G)_w=J(D)$ (see Eq.~\eqref{feb6-24}
above).
As $J(G)$ is Simis in degree $2$, by 
Corollary~\ref{simis-w}(a), so is $J(G)_w$. Thus, one has
$J(D)^2=J(D)^{(2)}$. 
\end{proof}

The following theorem characterizes the equality
of ordinary and symbolic powers of $I(D)$. 
Mandal and Pradhan showed that conditions (a) and (b) of
Theorem~\ref{In=I(n)} are sufficient conditions for the equality
of ordinary and symbolic powers of $I(D)$ \cite[Corollary 3.8]{Mandal-Pradhan}. 

\begin{theorem}\cite[Theorem~3.3]{weighted-symbolic}\label{In=I(n)}  
Let $D$ be a weighted oriented graph and let $G$ be its underlying 
graph. Then, $I(D)^n=I(D)^{(n)}$ for all $n\geq 1$ if and only if the following
two conditions hold:
\begin{enumerate}
\item[{\rm (a)}] Every vertex in $V^+(D)$ is a sink;
\item[{\rm (b)}] $G$ is a bipartite graph.
\end{enumerate}
\end{theorem}

\begin{remark}\label{jan6-24} If $J(D)^2=J(D)^{(2)}$, then
$I(D)^2=I(D)^{(2)}$ (Theorems~\ref{J2=J(2)} and \ref{In=I(n)}) but the
converse does not hold (Example~\ref{converseI2}).
\end{remark}

\begin{corollary}\label{ntf-weighted}  
Let $D$ be a weighted oriented graph and let $G$ be its underlying 
graph. If $I(D)$ has no embedded primes, then the following
conditions are equivalent:
\begin{enumerate}
\item[{\rm (a)}] ${\rm Ass}(J(D)^n)\subset{\rm Ass}(J(D))$ for all $n\geq
1$;
\item[{\rm (b)}] $J(D)^{n}=J(D)^{(n)}$ for all $n\geq 1$;
\item[{\rm (c)}] $I(D)^{n}=I(D)^{(n)}$ for all $n\geq 1$;
\item[{\rm (d)}] Every vertex in $V^{+}(D)$ is a sink and $G$ is
bipartite.
\end{enumerate}
\end{corollary}

\begin{proof} (a) $\Leftrightarrow$ (b) This follows from
Proposition~\ref{ntf-char-coro}. 

(c) $\Leftrightarrow$ (d) This follows from Theorem~\ref{In=I(n)}.

(b) $\Rightarrow$ (d)  Since $J(D)^2=J(D)^{(2)}$, by Theorem~\ref{J2=J(2)}, we
get that (d) holds. 

(d) $\Rightarrow$ (b) By Theorem~\ref{ntf-dual}, $J(G)$ is Simis, and
since all vertices of $V^+(D)$ are sinks, by Eq.~\eqref{feb6-24},
$J(G)_w=J(D)$. Hence, by Corollary~\ref{simis-w}(b), $J(D)$ is Simis.
\end{proof}

\section{Example}\label{examples-section}

\begin{example}\label{jan21-24} The ideal
$\mathfrak{q}=(t_1^2,t_2^2,t_1t_2,t_1t_3+t_2t_4)$ is a 
$(t_1,t_2)$-primary ideal of the polynomial ring
$S=\mathbb{Q}[t_1,t_2,t_3,t_4]$ but
$t_1t_3+t_2t_4\notin\mathbb{Q}[t_1,t_2]$
(cf.~Proposition~\ref{stanprim-iff}). The ideal $\mathfrak{q}$ is 
the saturation of $J=(t_1^2,t_2^2,t_1t_3+t_2t_4)$ with respect to
$(t_1,t_2)$, that is, $\mathfrak{q}=JS_{(t_1,t_2)}\cap S$. The
saturation of $J$ with respect to $(t_1,t_2)$ was computed 
using Procedure~\ref{procedure2}.
\end{example}

\begin{example} Let $(t_1t_2)$ be the ideal of $K[t_1,t_2]$ generated
by $t_1t_2$ and let $w$ be the weighting function
$w(a_1,a_2)=(a_1+a_2,a_2)$. Then, $(t_1t_2)=(t_1)\cap(t_2)$, $t^{w(e_1)}=t_1$,
$t^{w(e_2)}=t_1t_2$, and 
$$ 
(t_1t_2)_w=(t_1^2t_2)\subsetneq(t_1)_w\cap(t_2)_w=(t_1)\cap(t_1t_2)=(t_1t_2).
$$
\end{example}

\begin{example}\label{ntf-simis-minass} Let $I$ be the ideal of $K[t_1,t_2,t_3]$ generated
by $\mathcal{G}(I)=\{t_1t_2^2,t_2t_3^2,t_3t_1^2\}$. The irreducible
decomposition of $I$ is given by 
$$
I=(t_1^2,\, t_2)\cap(t_1,\, t_3^2)\cap(t_2^2,\, t_3)\cap(t_1^2,\,
t_2^2,\, t_3^2).
$$
\quad Setting 
$t^a=(t_1^2t_2^2t_3^2)(t_2^{n-2}t_3^{2(n-2)})=t_1^2t_2^nt_3^{2(n-1)}$
for all $n\geq 2$ and $\mathfrak{m}=(t_1,t_2,t_3)$, one has $t^a\notin I^n$ and $(I^n\colon
t^a)=\mathfrak{m}$ for all $n\geq 2$, that is,
$$
{\rm Ass}(I^n)=\{(t_1,t_2),\, (t_1,t_3),\, (t_2,t_3),\,
\mathfrak{m}\}\ \mbox{ for all }\ n\geq 1.
$$
\quad By Proposition~\ref{ntf-char-coro}, Simis ideals are normally
torsion-free. The ideal $I$ is normally torsion-free but $I$ is not
a Simis ideal because the $n$-th symbolic power of $I$ is given by
$$
I^{(n)}=(t_1^2,\, t_2)^n\cap(t_1,\, t_3^2)^n\cap(t_2^2,\, t_3)^n
\ \mbox{ for all }\ n\geq 1.
$$ 
\quad The ideal $I$ is not normal and its integral closure is $\overline{I}=I+(t_1t_2t_3)$.
\end{example}

\begin{example}\label{converse-ntf-simis-ass} Let $I\subset
K[t_1,\ldots,t_5]$ be the monomial ideal 
given by 
\begin{align*}
I=&(t_1^2,t_2^2,t_3^2)\cap (t_3^2,t_4^2,t_5^2)\cap
(t_3,t_4^5)\cap(t_1^4,t_2^4,t_3^4,t_4^4,t_5^4)\\
=&(t_3^4, t_3^2t_5^4, t_3^2t_4^4, t_2^4t_3^2, t_1^4t_3^2, t_2^2t_3t_5^4,
t_1^2t_3t_5^4, t_2^4t_3t_5^2, t_1^4t_3t_5^2, t_2^2t_4^5, 
t_1^2t_4^5, t_2^2t_3t_4^4, t_1^2t_3t_4^4, t_2^4t_3t_4^2, t_1^4t_3t_4^2).
\end{align*}
\quad Since $\mathfrak{m}=(t_1,\ldots,t_5)\in{\rm Ass}(I)$, by
Corollary~\ref{feb11-24}, we have
$I^{\langle n\rangle}=I^n$ for all $n\geq 1$. The set of associated
primes of $I^2$ is given by 
$$
{\rm Ass}(I^2)=\{(t_3,t_4),(t_1,t_2,t_3),(t_3,t_4,t_5), (t_1,t_2,t_3,t_4), 
(t_1,t_2,t_3,t_4,t_5)\},
$$
and $I$ is not normally torsion-free because 
$(t_1,t_2,t_3,t_4)\in{\rm Ass}(I^2)\setminus{\rm Ass}(I)$.
\end{example}

\begin{example}\label{generalized-edge-ideal-example}
Let $W=(w_{i,j})$ be an $s\times s$ matrix with non-negative integer 
entries such that $w_{i,j}=0$ if and only if $w_{j,i}=0$. The
\textit{underlying graph} $G_W$ of $W$ has vertex set $t_1,\ldots,t_s$
and $\{t_i,t_j\}$ is an edge of $G_W$ if $w_{i,j}\neq 0$. 
The \textit{generalized edge ideal} $I(W)$ of $W$ is generated by the
set of all $t_i^{w_{i,j}}t_j^{w_{j,i}}$ such that $w_{i,j}\neq 0$. This
ideal was first introduced by Das \cite{Das-K}. 
The ideals $I=(t_1t_2,\, t_1^2t_3,\, t_2^2t_3)$ and 
$L=(t_1^2t_2^3,\, t_1^5t_3^4,\, t_2^7t_3^6)$ are the generalized
edge ideal associated to the matrices:
\begin{align*}
W_1=\left[
\begin{matrix}
0&1&2\cr
1&0&2\cr
1&1&0
\end{matrix}
\right], &\quad\quad\quad
W_2=\left[
\begin{matrix}
0&2&5\cr
3&0&7\cr
4&6&0
\end{matrix}
\right],
\end{align*}
respectively \cite{Das-K}. According to \cite[Theorem~3.10]{Das-K},
these ideals are Simis, that is, $I^n=I^{(n)}$ and $L^n=L^{(n)}$ for
all $n\geq 1$. Using Procedure~\ref{procedure1}, we obtain that 
the irreducible and primary decompositions of the ideals $I$ and $L$ are given by: 
\begin{align*}
I=&(t_1,t_3)\cap(t_1,t_2^2)\cap(t_1^2,t_2)\cap(t_2,t_3)=
(t_1,t_3)\cap(t_1,t_2)^2\cap(t_2,t_3),\\
J=&(t_1^2,t_3^6)\cap(t_1^2,t_2^7)\cap(t_1^5,t_2^3)\cap(t_2^3,t_3^4)=
(t_1^2,t_3^6)\cap(t_1^5,t_1^2t_2^3,t_2^7)\cap(t_2^3,t_3^4).
\end{align*}
\quad In particular $I$ and $L$ have no embedded primes. This example shows that the radical of a Simis ideal is not 
Simis in general and that Conjecture~\ref{conjecture-simis} fails if one drops the
condition that the irreducible decomposition is minimal.
\end{example}

\begin{example} Let $I$ be the ideal $(t_1^2t_2^2,\, t_1t_2t_3,\,
t_2t_3^2,\, t_3^2t_4^2)$. Using \textit{Macaulay}$2$
\cite{mac2}, we get
\begin{align*}
I=(t_1,t_3^2)\cap(t_2,t_3^2)\cap(t_1^2,t_3)\cap(t_2^2,t_3)\cap(t_2,t_4^2).
\end{align*}
\quad In particular ${\rm Ass}(I)=\{(t_1,t_3),(t_2,t_3),(t_2,t_4)\}$
and $I$ has no embedded primes. We verified that $I^{(n)}=I^n$ for
$i=1,\ldots,10$.
\end{example} 

\begin{example}\label{dual-d-simis-ce}
Let $S=K[t_1,\ldots,t_{12}]$ be a polynomial ring and let 
$\mathcal{C}$ be the clutter whose edge ideal is given by
$$
I=I(\mathcal{C})=(t_1t_2t_3t_4t_5,\, t_1t_6t_7t_8t_9,\,
t_2t_6t_{10}t_{11}t_{12}). 
$$
\quad The following are mutually disjoint minimal vertex covers of
$\mathcal{C}$ 
$$
C_1=\{t_1,t_{12}\},\, C_2=\{t_2,t_7\},\, C_3=\{t_3,t_9,t_{11}\},\,
C_4=\{t_4,t_6\},\,
C_5=\{t_5,t_8,t_{10}\},
$$
such that
$V(\mathcal{C})=\{t_i\}_{i=1}^{12}=\bigcup_{i=1}^5C_i$ and every
edge of $\mathcal{C}$ has the form $\{t_{i_1},\ldots,t_{i_5}\}$,
where $t_{i_j}\in C_j$ for all $j$. Using \textit{Macaulay}$2$
\cite{mac2}, we get 
${\rm Ass}(I_c(\mathcal{C})^{(5)})\subsetneq{\rm
Ass}(I_c(\mathcal{C})^{5})$. Thus, $I_c(\mathcal{C})^{(5)}\neq
I_c(\mathcal{C})^{5}$. Setting $d=5$, this proves that the converse 
Theorem~\ref{dual-d-simis} fails. One can also verify that  
$$
{\rm
Ass}(I_c(\mathcal{C}))^{n}=\{\mathfrak{m}\}\cup\{(t_1,t_2,t_3,t_4,t_5),\,
(t_1,t_6,t_7,t_8,t_9),\,
(t_2,t_6,t_{10},t_{11},t_{12})\}
$$
for $n=2,\ldots,5$, where $\mathfrak{m}=(t_1,\ldots,t_{12})$
\end{example}

\begin{example}\label{converseI2} 
Let $G$ be a $5$-cycle with vertices $t_1,\ldots,t_5$. As $G$ has no
triangles, one has $I(G)^2=I(G)^{(2)}$ 
\cite[Theorem~4.13]{symbolic-powers-survey} but $J(G)^2\neq
J(G)^{(2)}$ because $t_1\cdots t_5\in J(G)^{(2)}\setminus J(G)^2$. 
\end{example}

\begin{appendix}

\section{Procedures}\label{Appendix}

\begin{procedure}\label{procedure2}
Computing the saturation of an ideal with respect to a prime ideal, and
checking whether or not the ideal is saturated using \textit{Macaulay}$2$ 
\cite{mac2}. This procedure corresponds to
Example~\ref{jan21-24}. 
\begin{verbatim}
restart
load "SymbolicPowers.m2"
S=QQ[t1,t2,t3,t4]
J=ideal(t1^2,t2^2,t1*t3+t2*t4)
--checks whether or not J is primary 
isPrimary J
--computes the saturation of J^n with respect to 
--an associated prime p of J
h=(n,k)->localize(J^n,(ass(J))#k)
q=h(1,0)
isPrimary q
--computes the saturation of q^n with respect to 
--an associated prime p of q
f=(n,k)->localize(q^n,(ass(q))#k)
f(1,0)
--checks whether or not q is saturated
f(1,0)==q
\end{verbatim}
\end{procedure}

\begin{procedure}\label{procedure1}
Computing the $n$-th symbolic powers $I^{(n)}$ and $I^{\langle
n\rangle}$ of a monomial ideal $I$, its irreducible
decomposition, and a minimal primary decomposition using
\textit{Macaulay}$2$ 
\cite{mac2}. This procedure corresponds to
Example~\ref{generalized-edge-ideal-example}. One can compute other examples by changing the
polynomial ring $S$ and the generators of the ideal $I$.
\begin{verbatim}
restart
load "SymbolicPowers.m2"
S=QQ[t1,t2,t3]
I=monomialIdeal(t1*t2,t1^2*t3,t2^2*t3)
--computes the associated primes of I
ass I
--computes the irreducible decomposition of I
irreducibleDecomposition(I)
--computes a minimal primary decomposition of I
primaryDecomposition I
n=2
--computes I^{<n>} using Ass(I)
symbolicPower(I,n)
--computes I^{(n)} using minAss(I)
symbolicPower(I,n,UseMinimalPrimes=>true)
--checks whether or not equality holds
symbolicPower(I,n)==I^n
apply(1..3,n->symbolicPower(I,n)==I^n)
--checks whether or not equality holds
symbolicPower(I,n,UseMinimalPrimes=>true)==I^n
apply(1..3,n->symbolicPower(I,n,UseMinimalPrimes=>true)==I^n)
mingens(symbolicPower(I,n,UseMinimalPrimes=>true)/I^n)
\end{verbatim}
\end{procedure}

\end{appendix}

\section*{Acknowledgments} 
We thank Mehrdad Nasernejad for valuable comments and corrections 
on a preliminary version of the paper.
We used \textit{Normaliz} \cite{normaliz2} and \textit{Macaulay}$2$
\cite{mac2} to compute integral closures and symbolic powers of
monomial ideals.

\bibliographystyle{plain}

\end{document}